\documentclass[12pt,reqno]{article}

\usepackage[usenames]{color}
\usepackage{amssymb}
\usepackage{graphicx}
\usepackage{amscd}

\usepackage[colorlinks=true,
linkcolor=webgreen,
filecolor=webbrown,
citecolor=webgreen]{hyperref}

\definecolor{webgreen}{rgb}{0,.5,0}
\definecolor{webbrown}{rgb}{.6,0,0}

\usepackage{color}
\usepackage{fullpage}
\usepackage{float}

\usepackage{graphics,amsmath,amssymb}
\usepackage{amsthm}
\usepackage{amsfonts}
\usepackage{latexsym}

\newcommand{\seqnum}[1]{\href{http://oeis.org/#1}{\underline{#1}}}
\newcommand{\supp}{{\rm supp}}

\begin{document}

\vspace*{2.1cm}

\theoremstyle{plain}
\newtheorem{theorem}{Theorem}
\newtheorem{corollary}[theorem]{Corollary}
\newtheorem{lemma}[theorem]{Lemma}
\newtheorem{proposition}[theorem]{Proposition}
\newtheorem{obs}[theorem]{Observation}
\newtheorem{claim}[theorem]{Claim}

\theoremstyle{definition}
\newtheorem{definition}[theorem]{Definition}
\newtheorem{example}[theorem]{Example}
\newtheorem{remark}[theorem]{Remark}
\newtheorem{conjecture}[theorem]{Conjecture}

\begin{center}

{\LARGE\bf Universal updates of Dyck-nest signatures}

\vskip 1cm
\large
Italo J. Dejter

University of Puerto Rico

Rio Piedras, PR 00936-8377

\href{mailto:italo.dejter@gmail.com}{\tt italo.dejter@gmail.com}
\end{center}

\begin{abstract}
Let $0<k\in\mathbb{Z}$.
 The anchored Dyck words of length $n=2k+1$ (obtained by prefixing a 0-bit to each Dyck word of length $2k$ and used to reinterpret the Hamilton cycles in the odd graph $O_k$ and the middle-levels graph $M_k$ found by M\"utze et al.) represent in $O_k$ (resp., $M_k$) the cycles of an $n$- (resp., $2n$-) 2-factor and its cyclic (resp., dihedral) vertex classes, and are equivalent to Dyck-nest signatures. A sequence is obtained by updating these signatures according to the depth-first order of a tree of restricted growth strings (RGS's), reducing the RGS-generation of Dyck words by collapsing to a single update the time-consuming $i$-nested castling used to reach each non-root Dyck word or Dyck nest.
This update is universal, for it does not depend on $k$.
\end{abstract}

\section{Introduction. Odd and middle-levels graphs}\label{s1}

Let $0<k\in\mathbb{Z}$, let $n=2k+1$ and let $O_k$ be the $k${\it -odd graph} \cite{B}, namely the graph whose vertices are the $k$-subsets of the cyclic group $\mathbb{Z}_n$ over the set $[0,2k]=\{0,1,\ldots, 2k\}$ and having an edge $uv$ for each two vertices $u,v$ if and only if $u\cap v=\emptyset$. The {\it characteristic vectors} of such subsets $u,v$ of $[0,2k]$ are the $n$-vectors $\vec{u},\vec{v}$ over $Z_2$ whose {\it supports} (i.e, the subsets of $[0,2k]$ composed by all nonzero entries of $u,v$), are exactly $u,v$, respectively. We may write $\vec{u},\vec{v}\in V(O_k)$, instead of $u,v\in V(O_k)$.
The set $V(O_k)$ of vertices of $O_k$ admits a partition into {\it cyclic classes} mod $n$, where two vertices $\vec{u},\vec{v}$ are in the same {\it class} if and only if they are related by a translation mod $n$, e.g., if $\vec{u}=u_0\cdots u_{2k}$, then $\vec{v}=u_iu_{i+1}\cdots u_{2k}u_0u_1\cdots u_{i-1}$, for some $i\in\mathbb{Z}_n=[0,2k]$. This is a translation that we denote by $i\in\mathbb{Z}_n$. The said cyclic classes mod $n$ are to be optionally used in our final result, Corollary~\ref{coro}.

We also consider
 the double covering graph $M_k$ of $O_k$, where $M_k$, referred to as {\it middle-levels graph}, is the subgraph of the Boolean lattice of subsets of $[0,2k]$ induced by the {\it levels} $L_k$ ($=V(O_k)$) and $L_{k+1}$, formed by the binary $n$-strings of weight $k$ and $k+1$, respectively  \cite{D2,D1,D3}.
 Two vertices $u\in L_k$ and $v\in L_{k+1}$ of $M_k$ are adjacent in $M_k$ if and only if $u\subset v$, with $u$ and $v$ taken as subsets of $[0,2k]$.
 The {\it double-covering graph map} $\Theta:M_k\rightarrow O_k$ restricts to the identity map over $L_k$ and to the {\it reversed complement} bijection $\aleph$ over $L_{k+1}$, that is: if $v\in L_{k+1}$ has characteristic vector $\vec{v}=v_0v_1\cdots v_{2k-1}v_{2k}$, then $\Theta(v)=\aleph(v)$ has characteristic vector $\bar{v}_{2k}\bar{v}_{2k-1}\cdots\bar{v}_1\bar{v}_0$ in $V(O_k)$, where $\bar{0}=1$ and $\bar{1}=0$.
To the partition of $V(O_k)$ into cyclic classes mod $n$, or $\mathbb{Z}_n$-classes, corresponds a partition of $V(M_k)=L_k\cup L_{k+1}$ into {\it dihedral classes}, or {\it $\mathbb{D}_n$-classes}, where $\mathbb{D}_n\supset\mathbb{Z}_n$ is the dihedral group of order $2n$.

An $n$-string $\Psi=0\psi_1\cdots\psi_{2k}$ in the alphabet $[0,n]$ in which each nonzero entry appears exactly twice is seen as a concatenation $W^i|X|Y|Z^i$ of substrings $W^i,X,Y$ and $Z^i$, where $W^i$ and $Z^i$ have length $i$, for some $0<i<k$.
In that case, the $n$-string $W^i|Y|X|Z^i$ is said to be a {\it $i$-nested castling} of $\Psi$ (time-consuming as it swaps parts of $\Psi$, with many position changes).

A {\it $k$-factor} of a graph $G$ is a spanning $k$-regular subgraph. A {\it $k$-factorization} is a partition of $E(G)$ into disjoint $k$-factors.
 A 2-factor (or {\it cycle factor} \cite{u2f}) in $O_k$ formed by $n$-cycles, with a pullback 2-factor in $M_k$ of $2n$-cycles via $\Theta^{-1}$, and used in constructing Hamilton cycles \cite{Hcs} and optionally in Corollary~\ref{coro} below, was analyzed in \cite{D3} from the viewpoint of restricted growth strings (RGS's \cite[p. \ 325]{Arndt}), which form the {\it RGS-tree} $\mathcal T$ of Lemma~\ref{infty}, below.

 In Section~\ref{s}, a modification of the arguments of \cite{D3} shows that such RGS's exert control over the Dyck paths of length $n$, that represent bijectively the cyclic (resp., dihedral) classes of vertices of $O_k$ (resp., $M_k$). These paths, viewed as {\it Dyck nests}, defined in Subsection~\ref{r4}, were related via the (time-consuming) $i$-nested castling operation controlled by the RGS-tree $\mathcal T$ that yields each non-root Dyck nest from its parent nest (\cite{D2,D1,D3}, or Theorem~\ref{thm1}) in the reinterpretation of the Hamilton cycle constructions in $O_k$ \cite{Hcs} and $M_k$ \cite{gmn,M}.

Such RGS-control will be reduced below, first by
viewing each Dyck nest as its {\it signature}, defined in Subsection~\ref{r} and shown to be equivalent to that Dyck nest in Theorem~\ref{signest}, and second by collapsing each $i$-nested castling to a {\it universal} single (one-step) update of the signature of each non-root Dyck nest from the signature of its parent nest in the RGS-tree $\mathcal T$. The term {\it universal}, introduced in Theorem~\ref{ell}, is taken in the sense that the integers representing such updates do not depend on the values of $k$, so that those integers are valid and unique for all concerned $O_k$'s and $M_k$'s.
The sequence formed by all such updates, controlled by the RGS-tree $\mathcal T$, is presented in Theorem~\ref{alfin}, accompanied by the sequence of their corresponding locations in Corollary~\ref{coro}, leading to its asymptotic analysis (Subsection~\ref{Remarque}).

\section{Restricted growth strings and i-nested castling}\label{s}

The {\it $k$-th Catalan number} \cite{oeis} \seqnum{A000108} is given by $C_k=\frac{(2k)!}{k!(k+1)!}$. Let $\mathcal S$ be the {\it sequence of RGS's} \cite{oeis} \seqnum{A239903}. It was shown in \cite{D2,D1} that the first $C_k$ terms of $\mathcal S$ represent both the Dyck words of length $2k$ and the {\it extended} Dyck words of length $n$, obtained by prefixing a 0-bit to each Dyck word, and yielding a sole corresponding Dyck path (Subsection~\ref{r4}).

The sequence ${\mathcal S}=(\beta(i))_{0\le i\in\mathbb{Z}}$ starts as ${\mathcal S}=(\beta(0),\ldots,\beta(17),\ldots)=$ $$(0,1,10,11,12,100,101,110,111,112,120,121,122,123,1000,1001,1010,1011,\ldots)$$ and has the lengths of any two contiguous terms $\beta(m-1)$ and $\beta(m)$, ($1\le m\in\mathbb{Z}$), constant unless $m=C_k$, for some $k>1$, in which case $\beta(m-1)=\beta(C_k-1)=12\cdots k$ has length $k$, and $\beta(m)=\beta(C_k)=10^k=10\cdots 0$ has length $k+1$.

To work in middle-levels and odd graphs in relation to their Hamilton cycles~\cite{gmn,M,Hcs}, RGS's were tailored as {\it germs} in \cite{D2,D1,D3}. A $k$-{\it germ} ($k>1$) is a $(k-1)$-string $\alpha=a_{k-1}a_{k-2}\cdots a_2a_1$ such that:

\begin{enumerate}
\item[\bf(a)] the leftmost position of $\alpha$,  namely position $k-1$, contains the entry $a_{k-1}\in\{0,1\}$;
\item[\bf(b)] given $1<i<k$, the entry $a_{i-1}$ at position $i-1$ satisfies $0\le a_{i-1}\le a_i+1$.
\end{enumerate}

Each RGS $\beta=\beta(m)$, where $0\le m\in\mathbb{Z}$, is transformed, for every $k\in\mathbb{Z}$ such that $k\ge$ length$(\beta)$, into a $k$-germ $\alpha=\alpha(\beta,k)=\alpha(\beta(m),k)$ by prefixing $k-$ length$(\beta)$ zeros to $\beta$.

Every $k$-germ $a_{k-1}a_{k-2}\cdots a_2a_1$ yields the $(k+1)$-germ $0a_{k-1}a_{k-2}\cdots$ $a_2a_1$.
A {\it non-null} RGS is obtained by stripping a $k$-germ $\alpha=a_{k-1}a_{k-2}\cdots a_2a_1$ $\ne 00\cdots 0$ of all the zeros to the left of its leftmost position containing a 1. We denote such an RGS still by $\alpha$,
say that the {\it null} RGS $\alpha=0$ represents all null $k$-germs $\alpha$, ($0<k\in\mathbb{Z}$), and use $\alpha=\alpha(m)$, or $\beta=\beta(m)$, both for a $k$-germ and for its corresponding RGS. In fact, $\alpha=\alpha(m)$, or $\beta=\beta(m)$, will be considered to be the RGS representing all the $k$-germs $\alpha=\alpha(m)$, or $\beta=\beta(m)$, respectively, ($0<k\in\mathbb{Z}$) leading to $\alpha$, or $\beta$, as an RGS, by stripping their zeros as indicated.

If $a,b\in\mathbb{Z}$, then let

\begin{enumerate}
\item[\bf(1)] $[a,b]=\{j\in\mathbb{Z} ; a\le j\le b\}$; \hspace{1cm}{\bf(2)} $[a,b[=\{j\in\mathbb{Z} ; a\le j<b\}$;
\item[\bf(3)] $]a,b]=\{j\in\mathbb{Z} ; a<j\le b\}$; \hspace*{1cm}{\bf(4)} $]a,b[=\{j\in\mathbb{Z}; a<j<b\}$.\end{enumerate}

Given two $k$-germs $\alpha=a_{k-1}\cdots a_1$ and $\beta=b_{k-1}\cdots b_1,$ where $\alpha\ne \beta$, we say that $\alpha$ precedes $\beta$, written $\alpha<\beta$, whenever either

\begin{enumerate}
\item[\bf (i)] $0=a_{k-1} < b_{k-1}=1$ or
\item[\bf (ii)] $\exists i\in[1,k[$ such that $a_i < b_i$ with $a_j=b_j$, $\forall j\in]i,k[$.
\end{enumerate}

 The resulting order of $k$-germs yields a bijection from $[0,C_k[$ onto the set of $k$-germs that assigns each $m\in[0,C_k[$ to a corresponding $k$-germ $\alpha=\alpha(m)$. In fact, there are exactly $C_k$ $k$-germs $\alpha=\alpha(m)<10^k$, $\forall k>0$. Moreover, we have the following trees $\mathcal T_k$, correspondences $F(\cdot)$ and RGS-tree $\mathcal T$ (this one, partially exemplified in display (\ref{treetau}) via its section for $k\le 5$).

\section{Ordered trees of k-germs and Dyck words}

We recall from \cite[Theorem 3.1]{D2} or \cite[Theorem 1]{D1} that the $k$-germs are the nodes of an ordered tree ${\mathcal T}_k$ rooted at $0^{k-1}$ and such that each $k$-germ $\alpha=a_{k-1}\cdots a_2a_1\ne0^{k-1}$ with rightmost nonzero entry $a_i$  ($1\le i=i(\alpha)<k$) has parent $\beta(\alpha)=b_{k-1}\cdots b_2b_1\!<\alpha$  in ${\mathcal T}_k$ with $b_i= a_i-1$ and $a_j=b_j$, for every $j\ne i$ in $[1,k-1]$.

\begin{lemma}\label{infty} By considering $k$-germs as RGS's, an infinite chain ${\mathcal T}_2\subset{\mathcal T}_3\subset\cdots\subset{\mathcal T}_k\subset\cdots$ of finite trees converges to their union, the RGS-tree $\mathcal T$.\end{lemma}

\begin{proof} Iterative inclusion of the successive trees $\mathcal T_k$ tends to the RGS-tree, as $k$ converges to infinity, where the original $k$-germs are considered as RGS as indicated.\end{proof}

\begin{theorem}\label{thm1}
To each $k$-germ $\alpha=a_{k-1}\cdots a_1$ corresponds an $n$-string $F(\alpha)$ with initial entry $0$ and having each $j\in[1,k]$ as an entry exactly twice. Moreover, $$F(0^{k-1})=012\cdots(k-2)(k-1)kk(k-1)\cdots 21,\;(e.g.,\; F((0)=011,\; F(00)=01221).$$ Furthermore, if $\alpha\ne 0^{k-1}$, let
\begin{enumerate}
\item $W^i$ and $Z^i$ be the leftmost and rightmost, respectively, substrings of length $i=i(\alpha)$ in $F(\beta)$, where $\beta$ is the parent of $\alpha$ in $\mathcal T_k$;
\item $c>0$ be the leftmost entry of $F(\beta)\setminus(W^i\cup Z^i)$, and
\item $F(\beta)\setminus(W^i\cup Z^i)$ be the concatenation $X|Y$, where $Y$ starts at the entry $c+1$ of $F(\beta)$.
\end{enumerate}
\noindent then $F(\alpha)=W^i|Y|X|Z^i$ is the $i$-nested castling of $F(\beta)=W^i|X|Y|Z^i$. In addition, $W^i$ is an ascending $i$-substring, $Z^i$ is a descending $i$-substring, and $kk$  is a substring of $F(\alpha)$.
\end{theorem}

\begin{proof} The proof is a slight modification of that of \cite[Theorem 3.2]{D2} or \cite[Theorems 2]{D1}, where the rightmost appearances of each integer of $[1,k]$ in every $F(\alpha)$ as in the statement were given as asterisks, *, or in \cite[Theorem 2]{D3} as equal signs, =.
\end{proof}

The disposition of RGS's in an initial section of the RGS-tree of Lemma~\ref{infty} (for $k\le 5$) is shown in display (\ref{treetau}), where the children of an RGS $\alpha$ at any level are disposed from left to right in the subsequent level, starting just below $\alpha$:

\begin{align}\label{treetau}
\begin{array}{rrrrrrrrrrrrrr}
0&&&&&&&&&&&&&\\
1&10&100&&&1000&&&&&&&&\\
  &11&101&110&&1001&1010&1100&&&&&&\\
   &12&     &111&120&&1011&1101&1110&&1200&&&\\
   &    &     &112&121&&1012&&1111&1120&1201&1210&&\\
   &    &     &      &122&&&&1112&1121&&1211&1220&\\
   &    &     &      &123&&&&&1122&&1212&1221&1230\\
   &    &     &      &&&&&&1123&&&1222&1231\\
   &    &     &      &&&&&&&&&1223&1232\\
   &    &     &      &&&&&&&&&&1233\\
   &    &     &      &&&&&&&&&&1234\\
      \end{array}
\end{align}

\section{Dyck words, k-germs and 1-factorizations}\label{nat}

A {\it binary $k$-string} (or {\it $k$-bitstring} \cite{gmn,M,Hcs}) is a sequence of length $k$ whose terms are the digits 0, called 0-{\it bits}, and/or 1, called 1-{\it bits}, respectively. The {\it weight} of a binary $k$-string is its number of 1-bits.

In this work, a {\it Dick word of length} $2k$ is defined as a binary $2k$-string
of weight $k$ such that in every prefix the number of 0-bits is at least equal to the number of 1-bits
 (differing from the Dyck words of \cite{Hcs} in which the number of 1-bits is at least the number of 0-bits).

 The concept of {\it empty Dyck word}, denoted $\epsilon$, whose weight is 0, also makes sense in this context. We will present each Dyck word as its associated {\it anchored Dyck word}, obtained by prefixing a 0-bit to it. In particular, $\epsilon$ is represented by the anchored Dyck word 0.

For each $k$-germ $\alpha$, where $k>1$, we define the binary string form $f(\alpha)$ of $F(\alpha)$ by replacing each first appearance of an integer $j\in[0,k]$ as an entry of $F(\alpha)$ by a 0-bit and the second appearance of $j$, in case $j\in[1,k]$, by a 1-bit (where 0-bits and 1-bits correspond respectively to the 1-bits and 0-bits used in \cite{Hcs}). Such $f(\alpha)$ is a binary $n$-string of weight $k$, namely an anchored Dyck word of length $n$ whose {\it support} $\supp(f(\alpha))$ is a vertex of $O_k$ and an element of $L_k$, while $\aleph(f(\alpha))$ is an element of $L_{k+1}$. Note that the pair $\{f(\alpha),\aleph(f(\alpha))\}$ together with the $\mathbb{Z}_n$-class of $f(\alpha)$ in $L_k$ ($=V(O_k)$) generate the $\mathbb{D}_n$-class of $f(\alpha)$ in $V(M_k)$. Thus, $f(\alpha)$ represents both a ${\mathbb Z}_n$-class of $V(O_k)$ and a ${\mathbb D}_n$-class of $V(M_k)$, which has Hamilton cycles lifted from those in $O_k$ \cite{Hcs,D3}, or independently, as in \cite{gmn,M,D2,D1}

\subsection{Dyck paths}\label{r4}

Each anchored Dyck word $f(\alpha)$ yields a {\it Dyck path} \cite{D3} obtained as a curve $\rho(\alpha)$ that grows from $(0,0)$ in the Cartesian plane $\Pi$ via the successive replacement of the 0-bits and 1-bits of $f(\alpha)$, from left to right, by {\it up-steps} and {\it down-steps}, namely segments $(x,y)(x+1,y+1)$ and $(x,y)(x+1,y-1)$, respectively. We assign the integers of the interval $[0,k]$ in decreasing order (from $k$ to 0) to the up-steps of $\rho(\alpha)$, from the top unit layer intersecting $\rho(\alpha)$ to the bottom one and from left to right at each concerning unit layer between contiguous lines $y,y+1\in\mathbb{Z}$, where $0\le y\in\mathbb{Z}$.
These assigned integers correspond to their leftmost appearances as entries of $F(\alpha)$.
Each leftmost appearance $j'$ of an integer $j\in[1,k]$ in $F(\alpha)$ corresponds to the starting entry of a Dyck subword $0u1v$ in $f(\alpha)$, where $u,v$ are Dyck subwords (possibly $\epsilon$). The Dyck subword $0u1v$ corresponds in $F(\alpha)$ to a substring $j'Uj''V$, where $U$ and $V$ correspond to $u$ and $v$, respectively, and $j''=j'\in[1,k]$.

\begin{figure}[htp]
\includegraphics[scale=0.47]{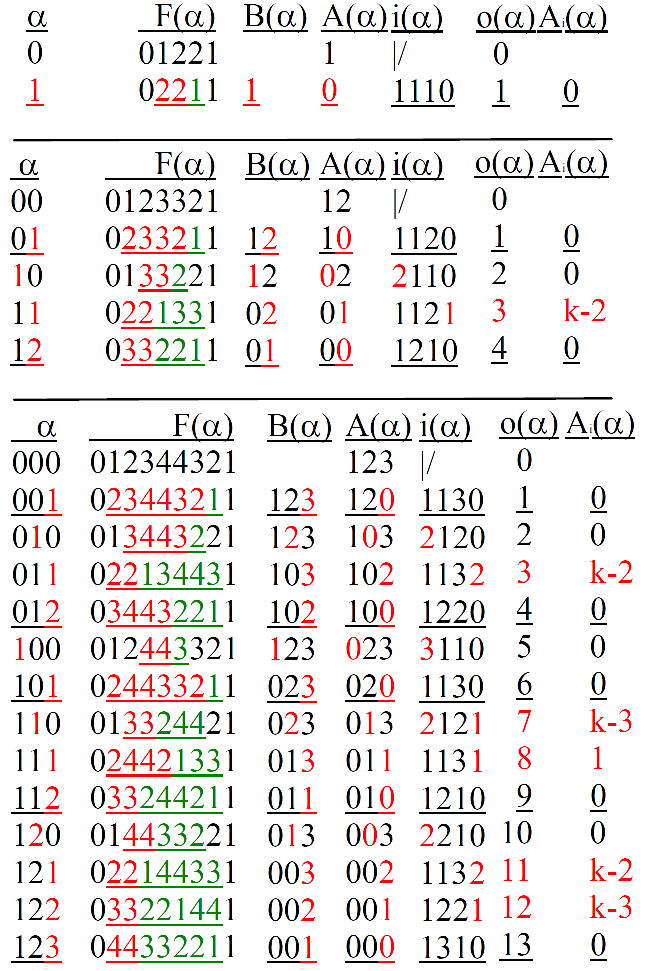}
\includegraphics[scale=0.49]{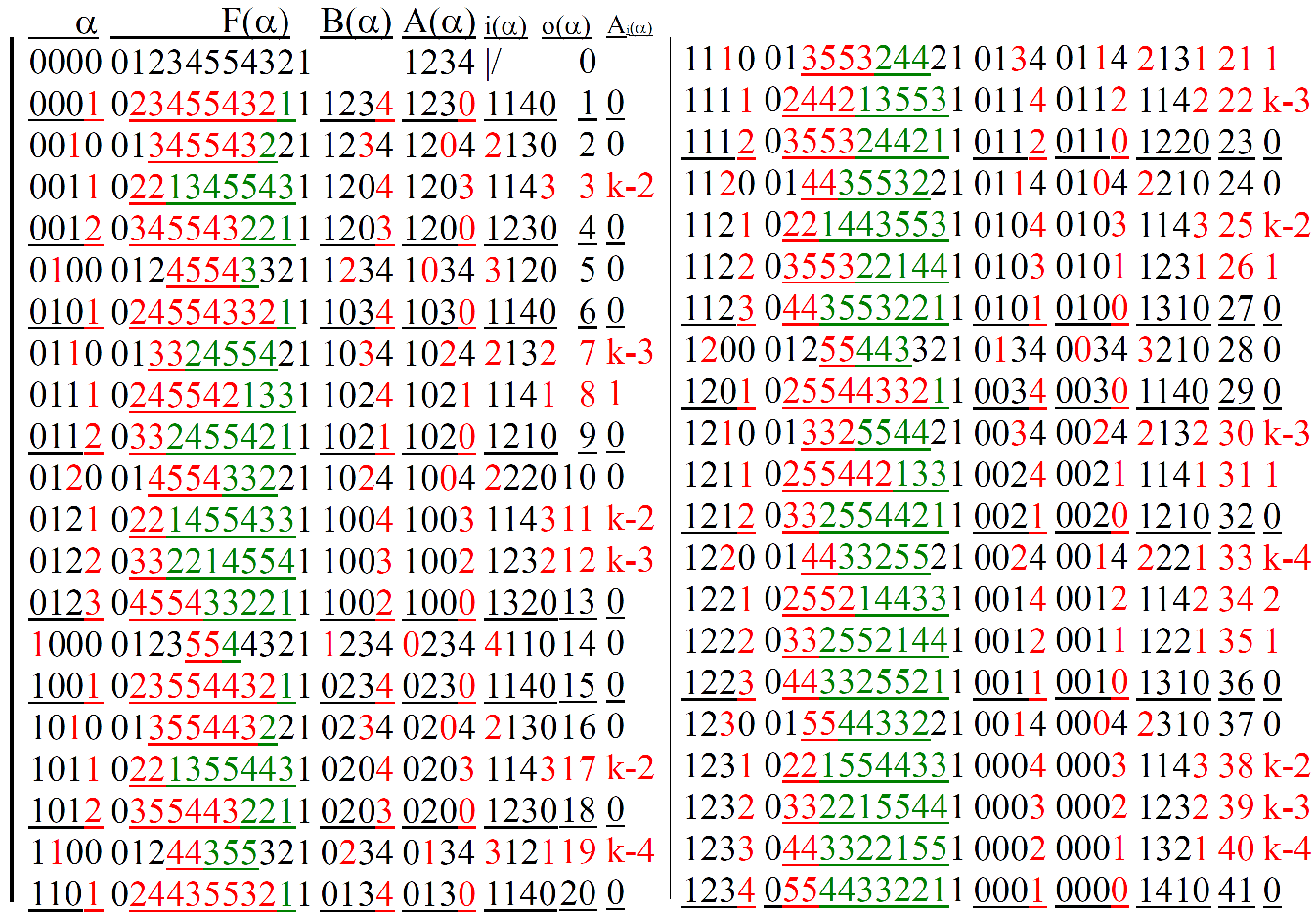}
\caption{List of $k$-germs $\alpha$, $n$-nests $F(\alpha)$, signatures and update entries, for $k=2,3,4,5$.}
\label{fig1}
\end{figure}

Each edge $uv$ of $O_k$ is taken as the union of a pair of {\it arcs} $\vec{uv}$ and $\vec{vu}$, that is a pair of oriented edges with {\it sources} $u$ and $v$ and {\it targets} $v$ and $u$, respectively.
Let us see that each first appearance of an integer $i\in[0,k]$ in $F(\alpha)$ (that we refer to as {\it color $i$}) determines uniquely an arc of $O_k$ and two edges of $M_k$.
The $n$-strings $F(\alpha)$ of Theorem~\ref{thm1} will be said to be {\it Dyck nests} of length $n$, or $n$-{\it nests}.
Say $u\in V(O_k)$ belongs to a Dyck nest $F(\alpha)$, seen as a $\mathbb{Z}_n$-class of $O_k$, and that $i'\in[0,k]$ is the first appearance of an integer $i$ in $F(\alpha)$. Then, there is a unique vertex $v$ in a $\mathbb{Z}_n$-class of $O_k$ corresponding to a Dyck nest $F(\alpha')$ such that $uv$ is an edge of $O_k$ and $u$ has its $i$-colored entry $i'$ in the same position as the entry with color $k-i$ in $v$, so we say that the {\it color of the arc} $\vec{uv}$ is $i$. In that case, the arc $\vec{vu}$ has color $k-i$, allowing to recover $u$ from $v$ as the unique vertex of $O_k$ such that to the entry of $v$ with color $k-i$ corresponds the entry in the same position in $u$ with color $k-(k-i)=i$. Thus, we say $\vec{uv}$ has color $i$ and $\vec{vu}$ has color $k-i$, this being the {\it supplementary color} of $i$ in $[0,k]$. The inverse images $\Theta^{-1}$ of $\vec{uv}$ and $\vec{vu}$ are formed by an arc from $L_k$ to $L_{k+1}$ and another arc from $L_{k+1}$ onto $L_k$ (see Example~\ref{ex}); they end up yielding a pair of edges in $M_k$.

\begin{example}\label{ex}
The translations $j\in\mathbb{Z}_n$ act on any anchored Dyck word $f(\alpha)$, yielding binary $n$-strings $f(\alpha).j$, so $f(\alpha).0=f(\alpha)$ itself.
This notation is also used for $n$-nests $F(\alpha)$.
Given $u=f(000).0=000001111\in O_4$, the arc color $i=3\in[0,4]$ determines an arc $\vec{uv}$ with source $u$ and target $v=f(001).5=111010000$. This information can be arranged as follows:

\begin{align}\label{display1}\begin{array}{|c|c|c|c|c|c|c|c|}
\alpha&j&F(\alpha).j&f(\alpha).j&O_k&^{L_4}\searrow\,_{L_{5}}&\aleph&^{L_{5}}\searrow\,_{L_4}\\\hline
000&0&012{\bf 3}44321&000{\bf 0}01111&u=5678&000{\bf 0}01111&\leftrightarrow&00001{\bf 1}111\\
001&5&432{\bf 1}10234&111{\bf 0}10000&v=0124&000{\bf 1}01111&\leftrightarrow&00001{\bf 0}111\\
\end{array}\end{align}
Display (\ref{display1}) shows from left to right: the 4-germs $\alpha$ for the source $u$ and target $v$ (columnwise) of the arc $\vec{uv}$; the corresponding translations $j\in\mathbb{Z}_9$; the $\mathbb{Z}_9$-translated Dyck nests $F(\alpha).j$, where the $i$-th entries are shown in bold trace; the $\mathbb{Z}_9$-translated anchored Dyck words $f(\alpha).j$, where the $i$-th entries are again shown in bold trace; and the two edges in the double covering $M_4$ of $O_4$ projecting onto $\vec{uv}$, which are related via $\aleph$.
\end{example}

\begin{figure}[htp]
\hspace*{10mm}
\includegraphics[scale=1.35]{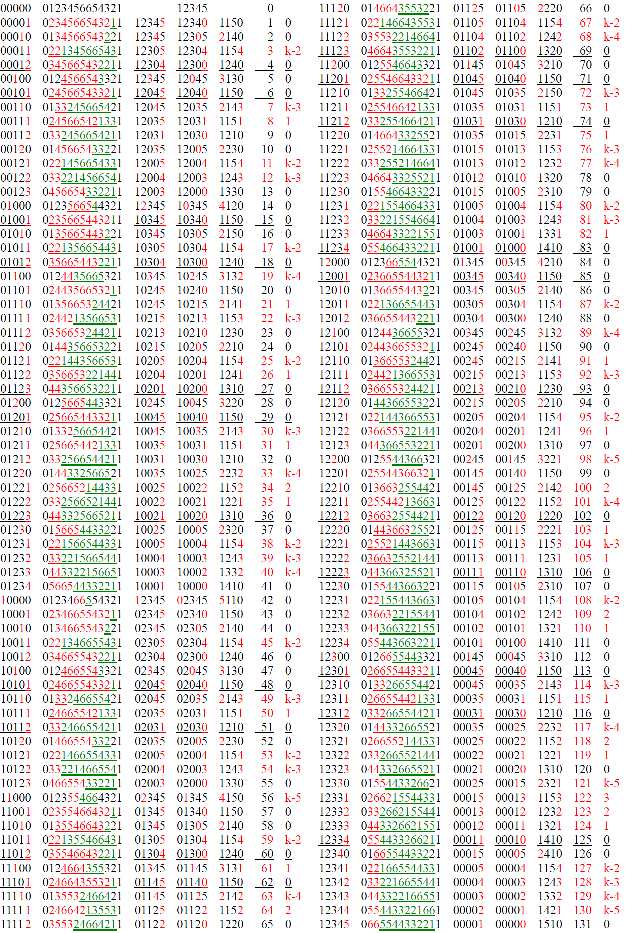}
\caption{List of $k$-germs $\alpha$, $n$-nests $F(\alpha)$, signatures and update entries, for $k=6$.}
\label{fig2}
\end{figure}

\subsection{Arc coloring and 1-factorizations}

Note that there is a coloring (or partition) of the set of arcs of $O_k$ resulting from Subsection~\ref{r4} and exemplified in Example~\ref{ex}. It induces a 1-factorization of $M_k$ into $(k+1)$ 1-factors, each formed by the edges whose arcs from $L_k$ to $L_{k+1}$ are colored with a corresponding integer of $[0,k]$. This factorization is known as the {\it modular} 1-factorization of $M_k$ \cite{D3}. In contrast, a different 1-factorization known as the {\it lexical} 1-factorization of $M_k$  \cite{gmn} exists. This is presented and exemplified in Example~\ref{kj}.

\begin{example}\label{kj}
Continuing as in Example~\ref{ex} but with $M_k$ rather than $O_k$, we modify and, instead of coloring with $k-i\in[0,k]$ the arc $\vec{uv}$ determined by the first appearance of $i\in[0,k]$ in the Dyck nest $F(\alpha)$ of each vertex $u$ of $M_k$ in $L_k$, we now color $\vec{uv}$ with $i\in[0,k]$, so that a 1-factorization of $M_k$ is determined, namely the lexical one \cite{gmn} mentioned above, with $\vec{vu}$ also colored with $i$. This is exemplified as follows, where $k=4$, color $i=3\in[0,4]$, and $\alpha=000$, so that $u=f(\alpha).0=f(\alpha)=000{\bf 0}01111$ (with the $i$-th entry in bold trace) is sent by $\aleph$  onto $\aleph(u)=00001{\bf 1}111\in L_5$:
\begin{align}\label{display2}\begin{array}{|c|c|c|c|c|c|c|}
V(M_4)&\alpha&j&F(\alpha).j&f(\alpha).j&\aleph&\aleph(f(\alpha).j)\\\hline
L_4&000&0&012{\bf 3}44321&000{\bf 0}01111&\leftrightarrow& 00001{\bf 1}111\in L_5\\
L_5&100&8&123{\bf 3}44210&000{\bf 1}01111&\leftrightarrow& 00001{\bf 0}111\in L_4\\
\end{array}\end{align}
In display (\ref{display2}), the corresponding edges from $u$ and $\aleph(u)$ end up onto $v=\aleph(w)=000{\bf 1}01111\in L_5$ and $w=\aleph^{-1}(v)=f(100).8=00001{\bf 0}111\in L_4$. These are the edges $uv=u\aleph(w)$ and $\aleph(u)v$
with both oppositely oriented arcs in each case having the same (lexical) color $i$, which differs with the modular-color situation in Subsection~\ref{r4} and Example~\ref{ex} (that is: with the colors $i$ and $k-i$ of the arcs of each edge differing as supplementary colors in $[0,k]$).
\end{example}

\section{Dyck nests and signatures}\label{DN&ts}

\begin{theorem}\label{t3} Each anchored Dyck word $w$ of length $n$ is the binary string $f(\alpha)$ associated to an $n$-nest $F(\alpha)$ obtained via the procedure of Theorem~\ref{thm1} from a specific $k$-germ $\alpha=\alpha(w)$.
\end{theorem}

\begin{proof}
The Lexical Procedure \cite[Section 7]{D2}, \cite[Section 7]{D1} restores the positive integer entries of $F(\alpha)$ corresponding to the $k$ non-initial 0-bits of $w=f(\alpha)$. These are the first appearances $j'$ of each integer $j\in[1,k]$ in $F(\alpha)$. By forming the Dyck word $0u1v$ of $f(\alpha)$, the second appearance $j''$ of $j$ is found by replacing its corresponding 1-bit in $f(\alpha)$ by $j=j''$ in $F(\alpha)$.
\end{proof}

\subsection{Dyck nests}\label{r5}

Our calling the strings $F(\alpha)$ by the name of {\it Dyck nests}, or {\it $n$-nests}, was suggested by the sets of nested intervals formed by the projections on the $x$-axis of the two appearances $j'$ and $j''$ of each integer $j\in[1,k]$ as numbers assigned to the respective up- and down-steps of each Dyck path $\rho(\alpha)$.

We take the tree ${\mathcal T}_k$ whose nodes were originally denoted via the $k$-germs $\alpha$, and denote them, further, via the $n$-nests $F(\alpha)$, in representation of the corresponding anchored Dyck words $f(\alpha)$. With this nest notation, $\mathcal T_k$ will be now said to be a {\it tree of Dyck nests}.

\begin{corollary}\label{t5} The set of $n$-nests $F(\alpha)$ is in one-to-one correspondence with the set of anchored Dyck words $f(\alpha)$ of length $n$.
\end{corollary}

\subsection{Signatures}\label{r}

Each $n$-nest $F(\alpha)$ is encoded by its {\it signature} $A(\alpha)=(A_{k-1}(\alpha),\ldots,A_2(\alpha),a_1)alpha)$, defined as
the vector of halfway-distance floors $A_j(\alpha)$ between the first ($j'$) and second ($j''$) appearances of each integer $j$ assigned to the respective up- and down-steps of the path $\rho(\alpha)$, where $k>j>0$. We write For example, if $j'k'k''j''$ (resp., $j'(k-1)'k'k''(k-1)''j''$) is a substring of $F(\alpha_1)$ (resp., $F(\alpha_2)$), then the halfway-distance floor of $j$ is $\lfloor d(j',j'')\rfloor=\lfloor 3/2\rfloor=1$ (resp. $\lfloor d(j',j'')\rfloor=\lfloor 5/2\rfloor=2$), engaged as the $j$-th entry of $A(\alpha_1)$ (resp., $A(\alpha_2)$).

\begin{claim}\label{cl} Using the equivalence of $n$-nests $F(\alpha)$ and signatures $A(\alpha)$ provided by Theorem~\ref{signest}, below, construction of the tree ${\mathcal T}_k$ of Dyck nests $F(\alpha)$ is simplified by updating just one entry of $A(\beta)$ to get $A(\alpha)$, instead of using the procedure in Theorem~\ref{thm1} to get $F(\alpha)$ from $F(\beta)$.
\end{claim}

\begin{example}\label{f1-2}
Claim~\ref{cl} is exemplified in Figures~\ref{fig1}--\ref{fig2} for $k=2,3,4,5,6$. In these figures, the first column for each such $k$ shows the $k$-germs $\alpha=a_{k-1}\cdots a_1$ in depth-first order of the node set of ${\mathcal T}_k$, in black except for $a_{i(\alpha)}$, which is in red; the second column shows the corresponding $n$-nests $F(\alpha)$ initialized in the top row as $F(0^{k-1})=$
$$012\cdots(k-2)(k-1)kk(k-1)(k-2)\cdots 21=01'2'\cdots(k-2)'(k-1)'k'k''(k-1)''(k-2)''\cdots 2''1''),$$ (with the ``prime'' notation after the equal sign in accordance to Subsection~\ref{r4}) and continued from the second row on as $F(\alpha)=W^i|Y|X|Z^i$, (as in Theorem~\ref{thm1}), where $W^i$ and $Z^i$ are in black, $Y$ is in red and $X$ is in green, and the parent $\beta$ of $\alpha$ in ${\mathcal T}_k$ having $F(\beta)=W^i|X|Y|Z^i$; this second column has the red-green numbers underlined;
the third and fourth columns have their rows as the  {\it signatures} $B(\alpha)=B_{k-1}B_{k-2}\cdots B_2B_1$ of $\beta$ (starting at the second row) and $A(\alpha)=A_{k-1}A_{k-2}\cdots A_2A_1$ of $\alpha$, specified by having $B_j=B_j(\alpha)$ and $A_j=A_j(\alpha)$, for each $j\in[1,k[$, as the numbers of pairs formed by the two appearances of each integer between the two appearances of $j$ in $F(\beta)$ and $F(\alpha)$, respectively; these third and fourth columns are determined by the black-red-green second column at each row; the fifth column, starting at the second row, is formed by four single-digit columns:
\begin{enumerate}
\item[(1)] the value $i=i(\alpha)$ in the current application of Theorem~\ref{thm1}; ($i$ in red if and only $i>1$);
\item[(2)] the corresponding value of $a_i=a_{i(\alpha)}$ in $\alpha=a_{k-1}a_{k-2}\cdots a_2a_1$;
\item[(3)] the corresponding value of $B_i(\alpha)=B_{i(\alpha)}(\alpha)$ in the third column;
\item[(4)] the value of $A_i(\alpha)=A_{i(\alpha)}(\alpha)$ in the fourth column, with $A_i$ in red if and if $A_i>0$;
\end{enumerate}
 the sixth column is the depth-first order $o(\alpha)$ of $\alpha$ in ${\mathcal T}_k$; all rows of the second column, below the first row, have the substring $kk$ (that is, $k'k''$, in terms of the appearances $k'$ and $k''$ of $k$) either in $Y$ (red) or in $X$ (green); after the initial black row $F(\alpha)=F(0^{k-1})$, the substring $kk$ is red in the two subsequent rows and becomes green in the fourth row; this corresponds to the red value $k-\ell=k-2$ of the seventh column.
For all columns but for the second one in Figures~\ref{fig1}-\ref{fig2}, each row which in the first column has $k$-germ $\alpha=a_{k-1}\cdots a_1$ with $a_1$
a local maximum (so that the following $k$-germ, say $\gamma=c_{k-1}\cdots c_1$, in the same first column, if any, has $c_1=0$) appears underlined.
\end{example}

\subsection{Role of substrings kk in Dyck nests}\label{obs13}

 Each value in the seventh column of Figures~\ref{fig1}-\ref{fig2} equals the corresponding value of item (4) in the fifth column, expressed in terms of the number $c$ of Theorem~\ref{thm1}, item 2, as:
 \begin{enumerate}
 \item[\bf(a)] $\ell$, if $kk$ is red, where $\ell$ is the number of green pairs $(j',j'')$ with $j>c$;
 \item[\bf(b)] $k-\ell$, if $kk$ is green, where $\ell$ is the sum of $c+1$ and the number $d$ of red pairs $(j',j'')$ with $j>c+1$.
 \end{enumerate}
 For example, all cases with $d>0$ (item (b)) in Figure~\ref{fig2} happen precisely for
 $$_{(\alpha,c,d)\;=\;(01111,2,1),\;(11110,3,1),\;(11122,3,1),\;(11221,2,1),\;(12111,2,1),\;(12211,2,2),\;(12221,2,1).}$$

Let $g$ be the correspondence that assigns the values $A_{i(\alpha)}(\alpha)$, (in the seventh column of Figures~\ref{fig1}-\ref{fig2}), to the orders $o(\alpha)$, (in the sixth column), where $\alpha$ refers to $k$-germs.

\begin{theorem}\label{id} For each $k$-germ $\alpha\ne 0^{k-1}$, the signatures $B(\alpha)$ and $A(\alpha)$ of the parent $\beta$ (of $\alpha$ in $\mathcal T_k$), and $\alpha$, respectively, differ solely at the $i(\alpha)$-th entry, that is: $$B_i(\alpha)=B_{i(\alpha)}(\alpha)\ne A_{i(\alpha)}(\alpha)=A_i(\alpha),\mbox{  while  } B_j(\alpha)=A_j(\alpha), \;\forall j\ne i=i(\alpha).$$
 \end{theorem}

 \begin{proof}
 There is a sole difference between the parent $\beta=b_{k-1}\cdots b_1$ of $\alpha=a_{k-1}\cdots a_1$ and $\alpha$ itself, occurring at the $i(\alpha)$-th position, whose entry is increased in one unit from $\beta$ to $\alpha$, that is: $a_{i(\alpha)}=b_{i(\alpha)}+1$. The effect of this on $F(\alpha)$, namely the $i$-nested castling of the inner strings $Y$ and $Z$ of $F(\beta)=X^i|Y|Z|W^i$ into $F(\alpha)=X^i|Z|Y|W^i$, modifies just one of the halfway-distance floors $A_j=\lfloor d(j',j'')/2\rfloor$ between the first appearance $j'$  of the corresponding $j\in[0,k[$ in $F(\alpha)$ and its second appearance, $j''$, namely $A_i=\lfloor d(i',i'')/2\rfloor$, where $i=i(\alpha)$.
 \end{proof}

\begin{theorem}\label{signest}
The correspondence that assigns each $n$-nest to its signature is a bijection.
\end{theorem}

\begin{proof} Let $\alpha=a_{k-1}\ldots a_2a_1$ be a $k$-germ. The $n$-nest $F(\alpha)=c_0c_1\ldots c_{2k}$ has rightmost entry $c_{2k}=1''$, so $A_1(\alpha)$ determines the position of $1'$. For example, if $A_1(\alpha)=0$, then
$c_{2k-1}=1'$, so $a_1$ is a local maximum (indicated in Figures~\ref{fig1}--\ref{fig2} by having $\alpha$, $B(\alpha)$, $A(\alpha)$, $\cdots$, $o(\alpha)$, $A_{i(\alpha)}(\alpha)$ underlined). To obtain $F(\alpha)$ from $A(\alpha)$, we initialize $F(\alpha)$ as the $n$-string $F^0=00\cdots 0$.
Setting the positions of $1'',1',2'',2',\ldots,(k-1)'',(k-1)'$ successively in place of the zeros of $F^0$ in their places from right to left according to the indications $A_1(\alpha)$, $A_2(\alpha)$, $\ldots$, $A_{k-1}(\alpha)$, is done in stages: first setting the pairs $(i',i'')$ as outermost pairs from right to left; when reaching the initial 0, we restart if necessary on the right again with the replacement of the remaining zeros by the remaining pairs $(i',i'')$ in ascending order from right to  left.
Thus, given $A(\alpha)$, we recover $F(\alpha)$. 
\end{proof}

\begin{example}
With $k=6$, $A(11111)=01122$, (resp.,  $A(12122)=00201$), we go from
$F^0$ to
$$\begin{array}{|l|l|l|}
0200002100001\mbox{ to}&\mbox{ (resp., }&0300003221001\mbox{ to}\\
0240042130031\mbox{ to}&&0366553221441\\
0236642135531&&\\
\end{array}\;),$$
the last row yielding four (resp., two) entries separating the two appearances $1'$ and $1''$ of $1\in[0,k]$, namely $3',5',5''$ and $3''$, (resp., $4'$ and $4''$).
\end{example}

Theorem~\ref{signest} provides a fashion of counting Catalan numbers
via RGS's \cite{D2,D1} different from that of \cite[item (u), p. \ 224]{Stanley}.
Both fashions, which are compared in \cite{D2}, accompany the counting list of RGS's in reversed order. In both cases (namely Theorem~\ref{signest} and item (u)),
the null root RGS, 0, corresponds to the signatures $12\cdots k$, for all $0<k\in\mathbb{Z}$; and the last RGS for every such $k$ corresponds to the signatures $0^k$. Thus, these initial (resp., terminal) terms coincide. However, these two counting lists with same initial (resp., terminal) terms differ in general.

\begin{theorem}\label{ell} \begin{enumerate}\item[(1)] The correspondence $g$, whose definition precedes Theorem~\ref{id}, is extended uniquely for each $k>1$ and $k$-germ $\alpha$, so that in terms of $\alpha$ seen as an RGS, the value of $g(o(\alpha))=A_{i(\alpha)}(\alpha)$ is expressible either as $\ell$ or as $k-\ell$, as in Subsection~\ref{obs13}.
\item[(2)] Registration of the value $\ell$ (resp., $-\ell$) at each stage in ${\mathcal S}\setminus\beta(0)$
for which $g(o(\alpha))$ is expressible as $\ell$ (resp., $k-\ell$) as in item (1), is performed independently of $k$, so it constitutes a universal single update of Dyck-nest signatures, just controlled by the RGS tree. This yields an integer sequence accompanying the natural order of RGS's in $\mathcal S$.    \end{enumerate}
\end{theorem}

The updates mentioned in Theorem~\ref{ell}, item (2), will be expressed in terms of the function in display (\ref{(1)}), to be employed in Theorems~\ref{1} and~\ref{j}, respectively.

\begin{proof}
The options in item (1) depend on whether the substring $k'k''$ lies in $Y$ (red) or in $X$ (green). In the first case, $g(o(\alpha))$ is of the form $\ell$. Otherwise, it is of the form $k-\ell$, for if $k$ is increased to $k+1$, then the substring $(k+1)'(k+1)''$ separates $k'$ and $k''$, thus adding one unit to $g(o(\alpha))$, so that $k-\ell$ becomes $(k+1)-\ell$. This happens independently of the values of $k$, yielding item (2).
\end{proof}

\begin{example}
The nonzero values $g(k)$ are initially as follows:
$g(3)=k-2$,
$g(7)=k-3$,
$k(8)=1$,
$g(11)=k-2$, $g(12)=k-3$, $g(17)=k-2$, $g(19)=k-4$, $g(21)=1$, $g(22)=k-3$, $g(25)=k-2$, $g(26)=1$, $g(30)=k-3$, $g(31)=1$, $g(33)=k-4$, $g(34)=2$, $g(35)=1$, $g(38)=k-2$, $g(39)=k-3$, $g(40)=k-4$, etc.
\end{example}

\begin{corollary}\label{cuc}
The following items hold:

\begin{enumerate}\item[{\bf(A)}]
The leftmost entry in the substring $W^i$ of $F(\alpha)=X^i|Z|Y|W^i$ is $i''$.

\item[{\bf(B)}] If the substring $k'k''$ of $F(\alpha)$ appears to the left of $i'$ in $F(\alpha)$, then $g(o(\alpha))$ equals the number of pairs $(j',j'')$ in the interval $]i',i''[$, for all pertaining integers $j\in[1,k[$. In particular, $F(\alpha)$ ends at the substring $1'1''$ if and only if $g(o(\alpha))=0$.

\item[{\bf(C)}] If $k'k''$ lies in $]i',i''[$ then $k'k''$ is contained in $X$ (green substring in $F(\alpha)$, Figures~\ref{fig1}--\ref{fig2}) and $g(o(\alpha))=k-j$, where $j=j(\alpha)$ is determined as follows: since $i(\beta)=1+i(\alpha)$, where $\beta=\beta(\alpha)$ is the parent of $\alpha$, then $j$ is the sum of $g(o(\beta))$ (which is as in item (B)) plus the leftmost red number of $F(\alpha)$.
\end{enumerate}
\end{corollary}

\begin{proof} The statement follows from Subsection~\ref{obs13} and Theorems~\ref{signest} and~\ref{ell}. In particular, items (B) and (C) are equivalent to items 1 and 2 of Subsection~\ref{obs13}, respectively.
\end{proof}

\begin{example}
Let $k=5$. Then, $g(21)=g(o(1110))=1$, as $]i',i''[=]2',2''[$ contains just the pair $(4',4'')$, accounting for one pair by Corollary~\ref{cuc}(B). For $\alpha=1111$, $k'k''$ is green and $g(22)=g(o(1111))=g(o(\alpha))=k-j=k-3$, where $j=3$ is the sum of $g(o(\beta))=g(o(1110))=g(21)=1$ and the leftmost red number of $F(\alpha)$, namely 2.
In addition, $g(28)=g(o(1200))=0$ has child $\alpha=1210$ with $g(o(\alpha))=g(30)=k-3$, because the leftmost red entry of $F(\alpha)$ is 3.
The child $\alpha'=1220$ of $\alpha$ has $g(o(\alpha'))=g(33)=k-(3+1)=k-4$.
However, the child $\alpha''=1230$ of $\alpha'$ has $g(o(\alpha''))=g(37)=0$.
Now, the child $1211$ of $\alpha$ has $g(o(1211))=1$, because $1'$ is the leftmost number of $W^1$ and there is only one pair of appearances of a member of $[1,k-1]=[1,4]$, namely $3'3''$,  between $1'$ and $1''$.
\end{example}

\section{Universal single updates}

Now, we introduce strings $A_i^j$, for all pairs $(i,j)\in\mathbb{Z}^2$ with $1<i\le j$. The entries of each $A_i^j$ are integer pairs $(\iota,\zeta)$, denoted $\iota_\zeta$, starting with $1_1$, initial case of the more general notation $1_j$, for $j\ge 1$. The strings $A_i^j$  are conceived as shown in Table~\ref{tab1}. The components $\iota$ in the entries $\iota_\zeta$ represent the indices $i=i(\alpha)$ of Theorem~\ref{thm1} in their order of appearance in $\mathcal S$, and $\zeta$ is an indicator to distinguish different entries $\iota_\zeta$ while $\iota$ is locally constant.

\begin{table}[htp]
$$\begin{array}{|l|}\hline
^{A_2^2=2_1|1_1|1_2;}
_{A_2^3=2_2|1_1|1_2|1_3;}\\
^{A_2^4=2_3|1_1|1_2|1_3|1_4;}
_{A_2^5=2_4|1_1|1_2|1_3|1_4|1_5;}\\
^{\cdots}
_{A_3^3=3_1|1_1|A_2^2|A_2^3=3_1|1_1|2_11_11_2|2_21_11_21_3;}\\
^{A_3^4=3_2|1_1|A_2^2|A_2^3|A_2^4=3_2|1_1|2_11_11_2|2_21_11_21_3|2_31_11_21_31_4;}
_{A_3^5=3_3|1_1|A_2^2|A_2^3|A_2^4|A_2^5=3_2|1_1|2_11_11_2|2_21_11_21_3|2_31_11_21_31_4|2_41_11_21_31_41_5;}\\
^{\cdots}
_{A_4^4=4_1|1_1|A_2^2|A_3^3|A_3^4=4_1|1_1|2_11_11_2|3_11_12_11_11_22_21_11_21_3|3_21_12_11_11_2|2_21_11_21_3|2_31_11_21_31_4;}\\
^{A_4^5=4_2|1_1|A_2^2|A_3^3|A_3^4|A_3^5;}
_{\cdots}\\
^{A_5^5=5_1|1_1|A_2^2|A_3^3|A_4^4|A_4^5;}
_{A_5^6=5_2|1_1|A_2^2|A_3^3|A_4^4|A_4^5|A_4^6;}\\
^{\cdots}
_{A_i^{i+j}=i_{i+j}|1_1|A_2^2|\cdots|A_{i-1}^{i-1}|A_{i-1}^i|\cdots|A_{i-1}^{i+j}, \; \forall 0<i\in\mathbb{Z}, \forall 0<j\in\mathbb{Z}.}\\\hline
\end{array}$$
\caption{Introduction of strings $A_i^j$, for all pairs $(i,j)\in\mathbb{Z}^2$ such that $1<i\le j$.}
\label{tab1}
\end{table}

Recalling items (B) and (C) of Corollary~\ref{cuc}, we define the updating integers $h(\alpha)$ by:
\begin{align}\label{(1)}h(\alpha)=\begin{cases}
g(o(\alpha)),&\mbox{ if }g(o(\alpha))\mbox{ is as in (B)}; \\
g(o(\alpha))-k,&\mbox{ if  }g(o(\alpha))\mbox{ is as in (C).}\\
\end{cases}
\end{align}

Next, consider the infinite string $A$ of integer pairs $i_\zeta$ formed as the concatenation
\begin{align}\label{(2)} A=A_1^1|A_2^2|\cdots|A_j^j|\cdots=*1_1|A_2^2|\cdots|A_j^j|\cdots,\end{align}
with $A_1^1=*|1_1=*1_1$ standing for the first two lines in tables as in Figures~\ref{fig1}--\ref{fig2}, where $*$, standing for the root of $\mathcal T_k$, represents the first such line, and $A_1^1$ represents the second one.

\begin{example} Illustrating (\ref{(2)}), Table~\ref{tab2} has its double-line heading formed by the subsequent terms of a suffix of $A$. The third heading line is formed first by the root $*$ of all trees $\mathcal T_k$ and then by the successive parameters $i=i(\alpha)>1$ initiating the substrings in the second line.
The fourth line contains the values $h(\alpha)$ for the parameters $i(\alpha)>1$ of the third line.
In every column, the values below that line are the values $h(\alpha)$ for RGS's $\alpha$ of the successive $k$-germs $\alpha$ with $i=i(\alpha)=1$.
Thus, below the third heading line, the values of each column represent the updates $h(\alpha)$ corresponding to all the maximal paths of trees $\mathcal T_k$ that, after its first node $\alpha$, has all other nodes $\alpha$ with $i=i(\alpha)=1$. Note that $A_1^1$ is represented as $[^{\,*}_{1_1}]$.
In the same way, we use notations $[^{\,3_1}_{1_1}]$ and $[^{\,4_1}_{1_1}]$, that could be generalized to $[^{\,j_1}_{1_1}]$.
\end{example}

Each prefix of $A$ corresponds to all $k$-germs representing a specific RGS $\alpha$ for increasing values of $k>1$, and is assigned the value $h(\alpha)$ to be its updating integer, in accordance to Corollary~\ref{cuc} but for the initial position, that is assigned an asterisk * to represent all the roots of the trees ${\mathcal T}_k$, for all $k>1$. More specifically, all prefixes of $A$ with Catalan-number lengths $C_k$ are the strings formed by locations $i=i(\alpha)$ in the natural order of the corresponding trees ${\mathcal T}_k$, while the values $h(\alpha)$ of the participating RGS's $\alpha$ occupy the subsequent positions down below the heading lines.

\begin{table}[htp]
$$\begin{array}{|r|r|rrr|rrrrrrrrr|}\hline
A_1^1&A_2^2&A_3^3&&&A_4^4&&&&&&&&\\\hline
[^{\;*}_{1_1}]&A_2^2&[^{\,3_1}_{1_1}]&A_2^2&A_2^3&[^{\,4_1}_{1_1}]&A_2^2&A_3^3&&&A_3^4&&&\\\hline
*&2&3&2&2&4&2&3&2&2&3&2&2&2\\\hline\hline
 *&\;{\it 0}&\;\;{\it 0}&{\it \mbox{-}3}&\;\;0&\;\;{\it 0}&\;\;0&{\it -4}&\;\;{\it 1}&\;\;0&\;\;0&{\it -3}&{\it -4}&\;\;0\\
 {\it 0}&{\it -2}&0&{\it 1}&-2&0&-2&0&{\it -3}&-2&0&1&{\it 2}&-2\\
 &0&&0&{\it-3}&&0&&0&{\it 1}&&0&{\it 1}&-3\\
 &&&&0&&&&&0&&&0&{\it -4}\\
 &&&& &&&&& &&&&0\\\hline
\end{array}$$
\caption{Exemplification of $A=A_1^1|A_2^2|\cdots|A_j^j|\cdots=*1_1|A_2^2|\cdots|A_j^j|\cdots$}
\label{tab2}
\end{table}

\begin{figure}[htp]
\includegraphics[scale=.782]{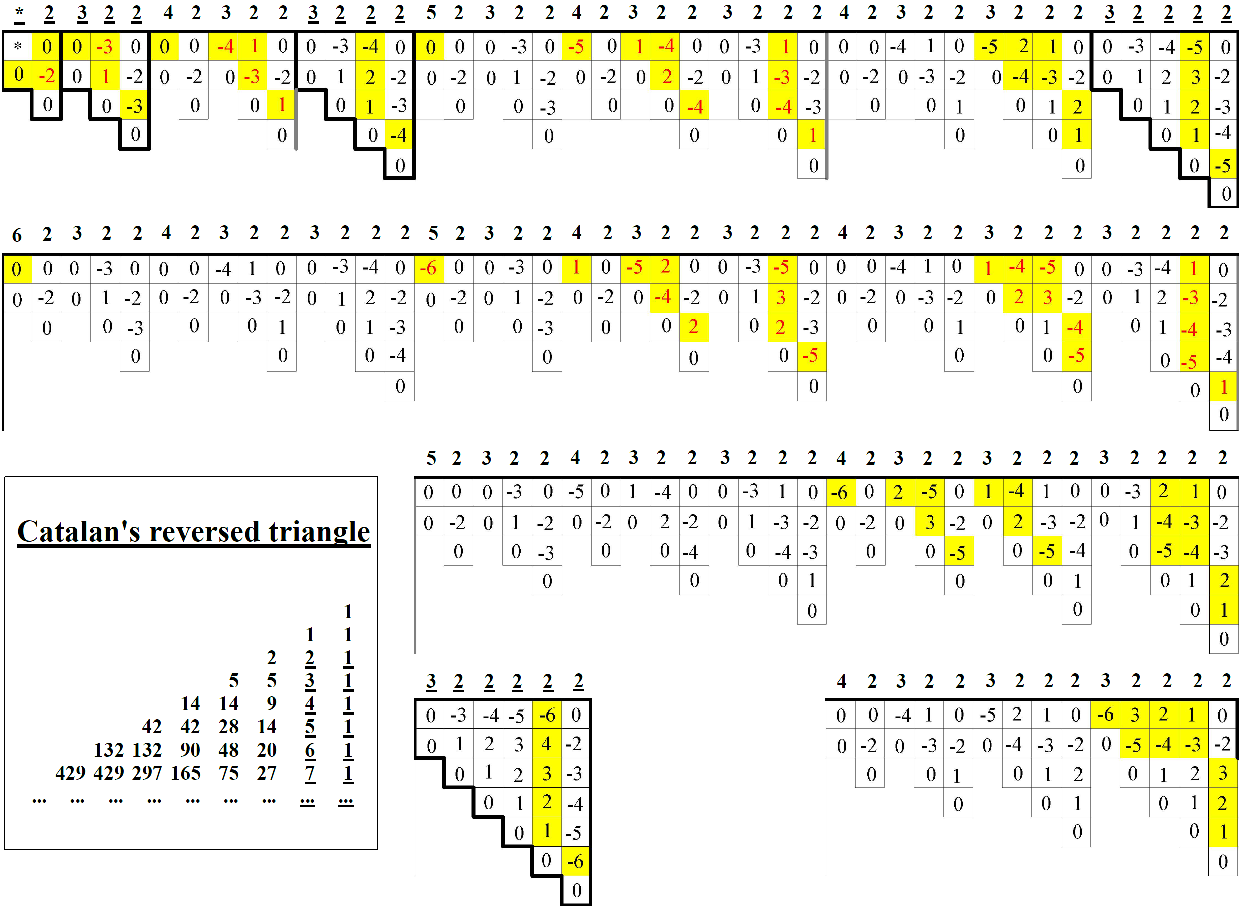}
\caption{Extension of Table~\ref{tab2} and partial view of $\Delta'$, for $k=2,3,4,5,6,7$}
\label{fig3}
\end{figure}

\begin{figure}[htp]
\includegraphics[scale=.725]{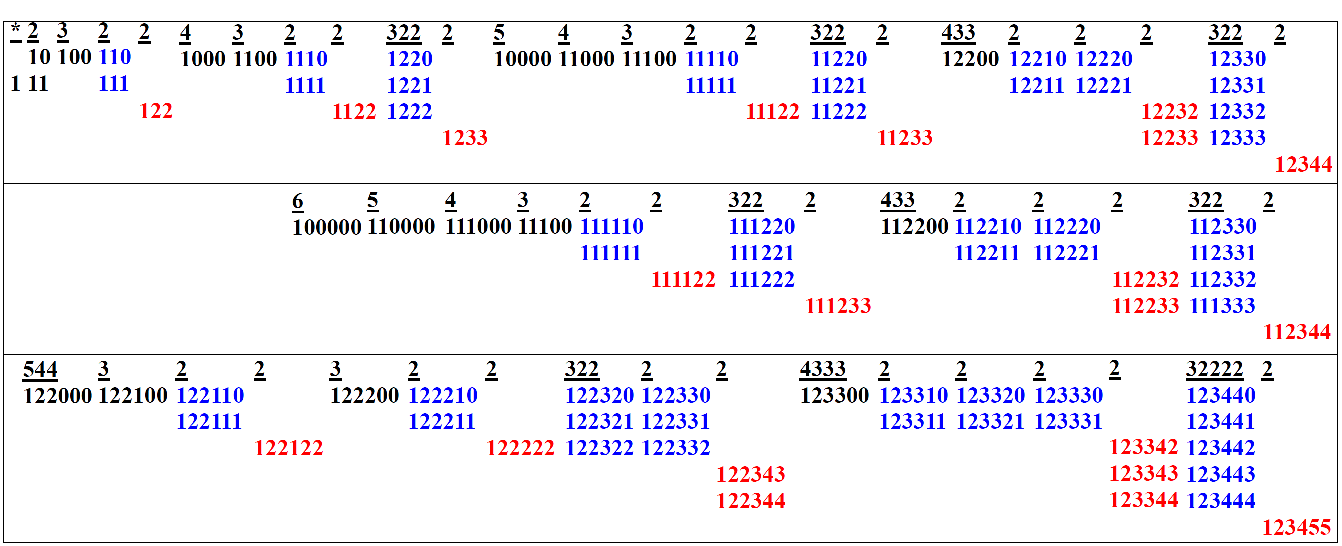}
\caption{Members of $\Phi_1$, for $k=2,3,4,5,6,7$}
\label{fig4}
\end{figure}

\begin{table}[htp]
$$\begin{array}{|ccccccccccccccc|}\hline
(1,0)\!\!\!&\!\!\!\!\!&\!\!\!&\!\!\!&\!\!\!&\!\!\!&\!\!\!&\!\!\!&\!\!\!&\!\!\!&\!\!\!&\!\!\!&\!\!\!&\!\!&\!\\
(2,0)\!\!\!&\!\!\!(1,\bar{2})\!\!&\!\!(1,0)\!&\!\!\!&\!\!\!&\!\!\!&\!\!\!&\!\!\!&\!\!\!&\!\!\!&\!\!\!&\!\!\!&\!\!\!&\!\!&\!\\\hline
(3,0)\!\!&\!\!(1,0)\!&\!\!\!&\!\!\!&\!\!\!&\!\!\!&\!\!\!&\!\!\!&\!\!\!&\!\!\!&\!\!\!&\!\!\!&\!\!\!&\!\!\!&\!\\
\!&\!(2,\bar{3})\!\!&\!\!(1,1)\!\!&\!\!(1,0)\!&\!\!\!&\!\!\!&\!\!\!&\!\!\!&\!\!\!&\!\!\!&\!\!\!&\!\!\!&\!\!&\!\!&\!\\
\!&\!\!&\!(2,0)\!\!&\!\!(1,\bar{2})\!\!&\!\!(1,\bar{3})\!\!&\!\!(1,0)\!&\!\!\!&\!\!\!&\!\!\!&\!\!\!&\!\!\!&\!\!\!&\!\!\!&\!\!\!&\!\\\hline
(4,0)\!\!&\!\!(1,0)\!&\!\!\!&\!\!\!&\!\!\!&\!\!\!&\!\!\!&\!\!\!&\!\!\!&\!\!\!&\!\!\!&\!\!\!&\!\!\!&\!\!\!&\!\\
\!&\!(2,0)\!\!&\!\!(1,\bar{2})\!\!&\!\!(1,0)\!&\!\!\!&\!\!\!&\!\!\!&\!\!\!&\!\!\!&\!\!\!&\!\!\!&\!\!\!&\!\!\!&\!\!\!&\!\\
\!&\!(3,\bar{4})\!\!&\!\!(1,0)\!\!\!&\!\!\!&\!\!\!&\!\!\!&\!\!\!&\!\!\!&\!\!\!&\!\!\!&\!\!\!&\!\!\!&\!\!\!&\!\!\!&\\
\!&\!\!&\!(2,1)\!\!&\!\!(1,\bar{3})\!\!&\!\!(1,0)\!&\!\!\!&\!\!\!&\!\!\!&\!\!\!&\!\!\!&\!\!\!&\!\!\!&\!\!\!&\!\!&\!\\
\!&\!\!\!&\!\!\!&\!(2,0)\!\!&\!\!(1,\bar{2})\!\!&\!\!(1,1)\!\!&\!\!(1,0)\!&\!\!\!&\!\!\!&\!\!\!&\!\!\!&\!\!\!&\!\!\!&\!\!&\!\\\hline
(5,0)\!\!&\!\!(1,0)\!&\!\!\!&\!\!\!&\!\!\!&\!\!\!&\!\!\!&\!\!\!&\!\!\!&\!\!\!&\!\!\!&\!\!\!&\!\!\!&\!\!\!&\!\\
\!\!\!&\!\!\!(2,0)\!\!&\!\!(1,\bar{2})\!\!&\!\!(1,0)\!&\!\!\!&\!\!\!&\!\!\!&\!\!\!&\!\!\!&\!\!\!&\!\!\!&\!\!\!&\!\!\!&\!\!\!&\!\\
\!&\!(3,0)\!\!&\!\!(1,0)\!&\!\!\!&\!\!\!&\!\!\!&\!\!\!&\!\!\!&\!\!\!&\!\!\!&\!\!\!&\!\!\!&\!\!\!&\!\!&\!\\
\!&\!\!&\!(2,\bar{3})\!\!&\!\!(1,1)\!\!&\!\!(1,0)\!&\!\!\!&\!\!\!&\!\!\!&\!\!\!&\!\!\!&\!\!\!&\!\!&\!&\!\!&\!\\
\!&\!\!\!&\!\!\!&\!(2,0)\!\!&\!\!(1,\bar{2})\!\!&\!\!(1,\bar{3})\!\!&\!\!(1,0)\!&\!\!\!&\!\!\!&\!\!\!&\!\!\!&\!\!\!&\!\!\!&\!\!&\\
\!&\!(4,\bar{5})\!\!&\!\!(1,0)\!&\!\!\!&\!\!\!&\!\!\!&\!\!\!&\!\!\!&\!\!\!&\!\!\!&\!\!\!&\!\!\!&\!\!\!&\!\!&\!\\
\!&\!\!&\!(2,0)\!\!&\!\!(1,\bar{2})\!\!&\!\!(1,0)\!&\!\!\!&\!\!\!&\!\!\!&\!\!\!&\!\!\!&\!\!\!&\!\!\!&\!\!\!&\!\!&\!\\
\!&\!\!&\!(3,1)\!\!\!&\!\!\!(1,0)\!&\!\!\!&\!\!\!&\!\!\!&\!\!\!&\!\!\!&\!\!\!&\!\!\!&\!\!\!&\!\!\!&\!\!\!&\!\\
\!&\!\!\!&\!\!\!&\!(2,\bar{4})\!\!&\!\!(1,2)\!\!&\!\!(1,0)\!&\!\!\!&\!\!\!&\!\!\!&\!\!\!&\!\!\!&\!\!\!&\!\!\!&\!\!\!&\!\\
\!&\!\!\!&\!\!\!&\!\!&\!(2,0)\!\!&\!\!(1,\bar{2})\!\!&\!\!(1,\bar{4})\!\!&\!\!(1,0)\!&\!\!\!&\!\!\!&\!\!\!&\!\!\!&\!\!\!&\!\!\!&\!\\
\!&\!\!\!&\!\!\!&\!(3,0)\!\!&\!\!(1,0)\!&\!\!\!&\!\!\!&\!\!\!&\!\!\!&\!\!\!&\!\!\!&\!\!\!&\!\!\!&\!\!&\!\\
\!&\!\!\!&\!\!\!&\!\!&\!(2,\bar{3})\!\!&\!\!(1,1)\!\!\!&\!\!\!(1,0)\!&\!\!\!&\!\!\!&\!\!\!&\!\!\!&\!\!\!&\!\!\!&\!\!&\!\\
\!&\!\!\!&\!\!\!&\!\!\!&\!\!\!&\!(2,1)\!\!&\!\!(1,\bar{3})\!\!&\!\!(1,\bar{4})\!\!\!&\!\!\!(1,0)\!&\!\!\!&\!\!\!&\!\!\!&\!\!\!&\!\!&\!\\
\!&\!\!\!&\!\!\!&\!\!\!&\!\!\!&\!\!\!&\!\!\!(2,0)\!\!&\!\!(1,\bar{2})\!\!&\!\!(1,\bar{3})\!\!\!&\!\!\!(1,1)\!\!\!&\!\!\!(1,0)\!&\!\!\!&\!\!\!&\!\!&\!\\
\!&\!\!&\!(4,0)\!\!&\!\!(1,0)\!&\!\!\!&\!\!\!&\!\!\!&\!\!\!&\!\!\!&\!\!\!&\!\!\!&\!\!\!&\!\!&\!\!&\!\\
\!&\!\!\!&\!\!\!&\!(2,0)\!\!&\!\!(1,\bar{2})\!\!&\!\!(1,0)\!&\!\!\!&\!\!\!&\!\!\!&\!\!\!&\!\!\!&\!\!\!&\!\!\!&\!\!\!&\!\!\\
\!\!\!&\!\!\!&\!\!&\!(3,\bar{4})\!\!\!&\!\!\!(1,0)\!&\!\!\!&\!\!\!&\!\!\!&\!\!\!&\!\!\!&\!\!\!&\!\!\!&\!\!\!&\!\!&\!\\
\!&\!\!\!&\!\!\!&\!\!&\!(2,1)\!\!\!&\!\!\!(1,\bar{3})\!\!\!&\!\!\!(1,0)\!&\!\!\!&\!\!\!&\!\!\!&\!\!\!&\!\!\!&\!\!\!&\!\!\!&\!\!\!\\
\!&\!\!\!&\!\!\!&\!\!\!&\!\!\!&\!\!(2,0)\!\!\!&\!\!\!(1,\bar{2})\!\!\!&\!\!\!(1,1)\!\!&\!\!(1,0)\!\!\!&\!\!\!&\!\!\!&\!\!\!&\!\!\!&\!\!\!&\!\\
\!\!\!&\!\!\!&\!\!\!&\!\!\!&\!(3,\bar{5})\!\!\!&\!\!\!(1,0)\!\!\!&\!\!\!&\!\!\!&\!\!\!&\!\!\!&\!\!\!&\!\!&\!\!&\!\!&\!\\
\!&\!\!\!&\!\!\!&\!\!\!&\!\!\!&\!\!(2,2)\!\!&\!\!(1,\bar{4})\!\!&\!\!(1,0)\!\!\!&\!\!\!&\!\!\!&\!\!\!&\!\!\!&\!\!\!&\!\!&\!\\
\!&\!\!\!&\!\!\!&\!\!\!&\!\!\!&\!\!&\!(2,1)\!\!&\!\!(1,\bar{3})\!\!&\!\!(1,1)\!\!&\!\!(1,0)\!&\!\!\!&\!\!\!&\!\!\!&\!\!\!&\!\\
\!&\!\!\!&\!\!\!&\!\!\!&\!\!\!&\!\!\!&\!\!\!&\!(2,0)\!\!&\!\!(1,\bar{2})\!\!&\!\!(1,2)\!\!&\!\!(1,1)\!\!&\!\!(1,0)\!\!\!&\!\!\!&\!\!&\!\\
\!&\!\!\!&\!\!\!&\!\!\!&\!\!\!&\!\!(3,0)\!\!\!&\!\!\!(1,0)\!\!\!&\!\!\!&\!\!\!&\!\!\!&\!\!\!&\!\!\!&\!\!&\!\!&\!\\
\!&\!\!\!&\!\!\!&\!\!\!&\!\!\!&\!\!&\!(2,\bar{3})\!\!&\!\!(1,1)\!\!&\!\!(1,0)\!\!\!&\!\!\!&\!\!\!&\!\!\!&\!\!&\!\!&\!\\
\!\!\!&\!\!\!&\!\!\!&\!\!\!&\!\!\!&\!\!\!&\!\!\!&\!\!\!(2,\bar{4})\!\!&\!\!(1,2)\!\!&\!\!(1,1)\!\!&\!\!(1,0)\!\!\!&\!\!\!&\!\!\!&\!\!\!&\!\\
\!\!\!&\!\!\!&\!\!\!&\!\!\!&\!\!\!&\!\!\!&\!\!\!&\!\!\!&\!(2,\bar{5})\!\!&\!\!(1,3)\!\!&\!\!(1,2)\!\!&\!\!(1,1)\!\!&\!\!(1,0)\!&\!\!&\!\\
\!&\!\!\!&\!\!\!&\!\!\!&\!\!\!&\!\!\!&\!\!\!&\!\!\!&\!\!\!&\!(2,0)\!\!&\!\!(1,\bar{2})\!\!&\!\!(1,\bar{3})\!\!&\!\!(1,\bar{4})\!\!&\!\!(1,\bar{5})\!\!&\!\!(1,0)\\\hline
\end{array}$$
\caption{Left-to-right list representation of  $\mathcal T_6$ whose nodes are pairs $(i(\alpha),h(\alpha))$ for the subsequent RGS's $\alpha$ in $\mathcal S$, and if some $h(\alpha)$ equals a negative integer $-\eta<0$, then it is shown as $\bar{\eta}$. The leftmost column shows the children of the root $(*,*)$ of $\mathcal T_6$.}
\label{tab3}
\end{table}

In  Figure~\ref{fig3}, the heading line of the top layer extends and continues the third heading line of Table~\ref{tab2}, its entries leading corresponding columns of values $h(\alpha)$, for $k<7$. This setting can be also seen as a left-to-right list representation of  $\mathcal T_6$ in Table~\ref{tab3}, whose nodes are pairs $(i(\alpha),h(\alpha))$ for the successive RGS's $\alpha$ in $\mathcal S$, where if some $h(\alpha)$ equals a negative integer $-\eta<0$, then is shown as $\bar{\eta}$, with the minus sign preceding $\eta$ shown as a bar over $\eta$. With such notation, the leftmost column of Table~\ref{tab3} shows the children of the root $(*,*)$ of $\mathcal T_6$. The adequately indented subsequent columns show the remaining descendant nodes at increasing distances from $(*,*)$. Also in Table~\ref{tab3}, horizontal lines separate the node sets of $\mathcal T_3-(*,*)$, $\mathcal T_4-\mathcal T_3$, $\mathcal T_5-\mathcal T_4$ and $\mathcal T_6-\mathcal T_5$.

By reading the entries of the successive columns of Table~\ref{tab2}, and more extensively in Figure~\ref{tab3}, etc., and then writing them from left to right, we obtain
the integer sequence $h(\mathcal S)$ formed by the values $h(\alpha)$ associated to the RGS's $\alpha$ of $\mathcal S$. For example, starting with Table~\ref{tab2}, we have that $h(\mathcal S)=(h(0),\ldots,h(41),\ldots)=$ $(*,\,0,\,0,-2,\,0,\,0,\,0,-3,\,1,\,0,\,0,-2,-3,\,0,\,0,\,0,$ $0,-2,0,\,-4,\,0,\,1,-3,\,0,\,0,-2,\,1,\,0,\,0,\,0,-3,\,1,\,0,-4,\,2,\,1,\,0,\,0,-2,-3,-4,\,0,\,\ldots).$

\subsection{Sequence of updates of Dyck-nest signatures}\label{touse}

The numbers in Italics in Table~\ref{tab2} initiate the subsequence $h(\Phi_1)$ of $h$-values of a subsequence $\Phi_1$ of $\mathcal S$, that will allow the continuation of the sequence of updates of the Dyck-nest signatures. These numbers reappear and are extended, in yellow squares in Figure~\ref{fig3}. Expressing $h(\Phi_1)$ with its initial terms as in Table~\ref{tab2}, we may write $h(\Phi_1)=(h(j) ; j=1,\,2,\,3,\,5,\,7,\,8,\,12,\,14,\,19,\,21,\,22,\,27,\,34,\,35,\,36,\,41,\,\ldots) =$ $(0,\,0,-2,\,0,-3,\,1,-3,$ $0,-4,\,1,-3,\,1,-4,\,2,\,1,-4,\,\ldots).$

In order to use $\Phi_1$, we recur to {\it Catalan's reversed triangle} $\Delta'$, whose initial lines, for $k=0,1,\ldots,7$, are shown on the lower left enclosure of Figure~\ref{fig3} and is obtained in general from Catalan's triangle $\Delta$ \cite{D2} by reversing its lines, so that with notation from \cite{D2}, the portion of $\Delta'$ shown in Figure~\ref{fig3} may be written as in Table~\ref{tab4}.

\begin{table}[htp]
$$\begin{array}{|c|c|c|c|c|c|c|c|c|}\hline
&&&&&&&&\tau_0^0=1\\
&&&&&&&\tau_1^1=1&\tau_0^1=1\\
&&&&&&\tau_2^2=2&\tau_1^2=2&\tau_0^2=1\\
&&&&&\tau_3^3=5&\tau_2^3=5&\tau_1^3=3&\tau_0^3=1\\
&&&&\tau_4^4=14&\tau_3^4=14&\tau_2^4=9&\tau_1^4=4&\tau_0^4=1\\
&&&\tau_5^5=42&\tau_4^5=42&\tau_3^5=28&\tau_2^5=14&\tau_1^5=5&\tau_0^5=1\\
&&\tau_6^6=132&\tau_5^6=132&\tau_4^6=90&\tau_3^6=48&\tau_2^6=20&\tau_1^6=6&\tau_0^6=1\\
&\tau_7^7=429&\tau_6^7=429&\tau_5^7=297&\tau_4^7=165&\tau_3^7=75&\tau_2^7=27&\tau_1^7=7&\tau_0^7=1\\
\cdots\cdots&\cdots\cdots&\cdots\cdots&\cdots\cdots&\cdots\cdots&\cdots\cdots&\cdots\cdots&\cdots\cdots&\cdots\cdots\\\hline
\end{array}$$
\caption{An initial detailed portion of Catalan's reversed triangle $\Delta'$.}
\label{tab4}
\end{table}

\subsection{Formations}\label{form}

Both in Table~\ref{tab2} and at the top layer of Figure~\ref{fig3}, we have the representations (to be called {\it formations}) of:

\begin{enumerate}
\item[{\bf(i)}] $(A_1^1)$, namely the leftmost column, (just $C_1=\tau_1^1=1$ columns), with $C_2=2$ entries;
\item[{\bf(ii)}] $(A_1^1|A_2^2)$, namely the $C_2=\tau_2^2=\tau_1^2=2$ leftmost columns, with a total of $C_3=5$ entries;
\item[{\bf(iii)}] $(A_1^1|A_2^2|A_3^3)$, namely the $C_3=\tau_3^3=\tau_2^3=5$ leftmost columns, with $C_4=14$ entries;
\item[{\bf(iv)}] $(A_1^1|A_2^2|A_3^3|A_4^4)$, namely the $C_4=\tau_4^4=\tau_3^4=14$ columns in Table~\ref{tab2} or the $C_4=14$ leftmost columns in Figure~\ref{fig3}, with a total of $C_5=42$ entries;
and
\item[{\bf(v)}] $(A_1^1|A_2^2|A_3^3|A_4^4|A_5^5)$, namely the top $C_5=\tau_5^5=\tau_4^5=42$ columns in Figure~\ref{fig3}, with a total of $C_6=132$ entries.\end{enumerate} These five formations correspond respectively to the trees $\mathcal T_2$, $\mathcal T_3$, $\mathcal T_4$, $\mathcal T_5$ and $\mathcal T_6$.
We subdivide the sets of respective columns according to the corresponding lines of $\Delta'$ considered as integer partitions $\Delta'_{k-2}$, namely: $\Delta'_0=(1)$, $\Delta_1=(1,1)$, $\Delta'_2=(2,2,1)$, $\Delta'_3=(5,5,3,1)$, $\Delta'_4=(14,14,9,4,1)$, and $\Delta'_5=(42,42,28,14,5,1)$ to be discussed subsequently.

Figure~\ref{fig3} contains the continuation for $k=7$ of the commented formations, extending the mentioned top layer of $\tau_5^5=42$ columns with a second and third layers (having $\tau_4^5=42$ and $\tau_3^5=28$ columns, respectively) and then with two additional parts in the fourth layer (having $\tau_2^5=14$ on the right, and $\tau_1^5+\tau_0^5=5+1$ columns on the left, respectively), and representing all of $\mathcal T_7$. These numbers of columns, namely (42,42,28,14,5,1), correspond to the sixth line $\Delta'_5$ of $\Delta'$, namely $\Delta'_5=(\tau_5^5,\tau_4^5,\tau_3^5,\tau_2^5,\tau_1^5,\tau_0^5)$.

\begin{figure}[htp]
\includegraphics[scale=.735]{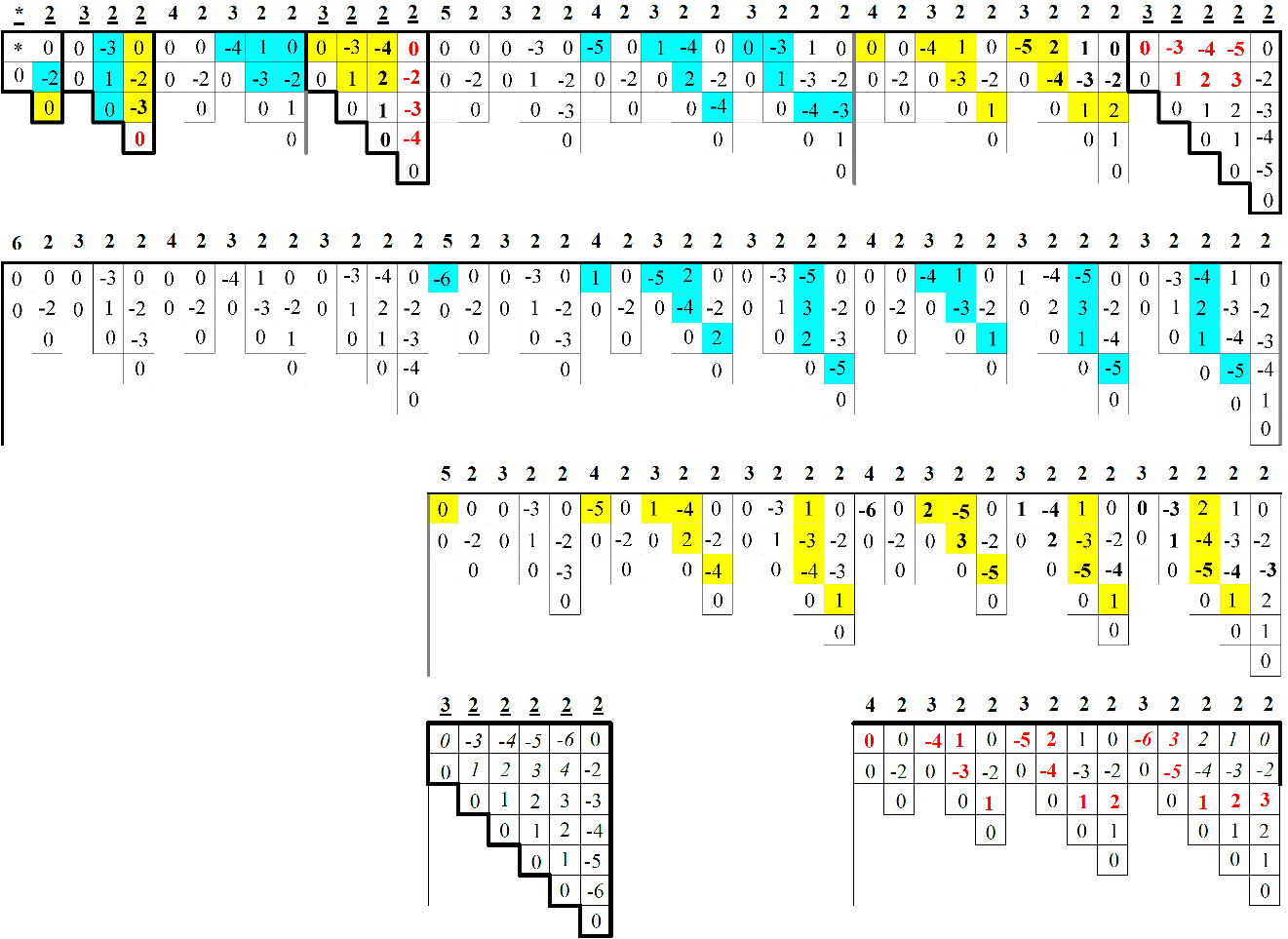}
\caption{Information for $\Phi_2,\Phi_3,\Phi_4,\Phi_5$}
\label{fig5}
\end{figure}

Still in Figure~\ref{fig3} for $\mathcal T_7$, the first $\tau_5^5=42$ columns (top layer) have lengths correspondingly equal to the lengths of the subsequent $\tau_4^5=42$ columns (second layer, delimited on the right by a thick gray vertical segment). Of these, the final 28 columns have lengths correspondingly equal to the lengths of the subsequent $\tau_3^5=28$ columns (third layer). Of these, the final 14 columns have lengths correspondingly equal to the lengths of the subsequent $\tau_2^5=14$ columns (fourth right layer). Of these, the final 5 columns have lengths correspondingly equal to the lengths of the subsequent $\tau_1^5=5$ columns (in the fourth left layer). It remains just $\tau_0^5=1$ column, formed by $k=7$ values of $h(\alpha)$. The said numbers of columns account for the partition
$\Delta'_5=(42,42,28,14,5,1)$, representing all the columns associated with the maximal paths of $\mathcal T_7$ formed by nodes associated with RGS's $\alpha$ with $i(\alpha)=1$.
Similar cases are easy to obtain in relation to $\mathcal T_k$, for $k<7$, where thick gray vertical segments delimit on the right the 14 (resp., 5) columns next to the first 14 (resp., 5) columns; (the same could have been done for the two columns next to the first two columns). A similar observation holds for every other row of $\Delta'$.

Some of the heading numbers in Figure~\ref{fig3} appear underlined, corresponding to the final $k-1=\tau_1^{k-2}+\tau_0^{k-2}=(k-2)+1$ columns for each exemplified $\mathcal T_k$. The resulting column sets appear encased with a thicker border.

\subsection{Main results}\label{5dogs}

The subsequence $\Phi_1$ of $\mathcal S$, a member of a family of subsequences $\{\Phi_j ; 1\le j\in\mathbb{Z}\}$ satisfying for $j>1$ the rules 1--3 below, is such that $i(\Phi_1)$ is the subsequence of $i(\mathcal S)$ formed by all indices $i(\alpha)$ larger than 1, exemplified in the heading line of Figure~\ref{fig3}. The mentioned rules 1--3 are as follows:

\begin{table}[htp]
$$\begin{array}{||r|r||r|r||r|r||r|r||r|r||r|r||}\hline
\alpha&h(\alpha)&\alpha_1&h(\alpha_1)&\alpha_2&h(\alpha_2)&\alpha_3&h(\alpha_3)&\alpha_4&h(\alpha_4)&\alpha_5&h(\alpha_5)\\\hline\hline
0&*&00&*&10&0&&&&&&\\
1&0&01&{\bf 0}&11&{\it -2}&12&{\bf 0}&&&&\\\hline
0&*&000&0&100&0&&&&&&\\
1&0&001&0&101&0&&&&&&\\
10&0&010&{\bf 0}&110&{\it -3}&120&{\bf 0}&&&&\\
11&-2&011&{\bf -2}&111&{\it 1}&121&{\bf -2}&&&&\\
12&0&012&0&112&0&122&1&123&0&&\\\hline
0&*&0000&*&1000&0&&&&&&\\
1&0&0001&0&1001&0&&&&&&\\
10&0&0010&0&1010&0&&&&&&\\
11&-2&0011&-2&1011&-2&&&&&&\\
12&0&0012&0&1012&0&&&&&&\\
100&0&0100&{\bf 0}&1100&{\it -4}&1200&{\bf 0}&&&&\\
101&0&0101&0&1101&0&1201&0&&&&\\
110&-3&0110&{\bf -3}&1110&{\it 1}&1210&{\it -3}&&&&\\
111&1&0111&{\bf 1}&1111&{\it -3}&1211&{\bf 1}&&&&\\
112&0&0112&0&1112&0&1212&0&&&&\\
120&0&0120&0&1120&0&1220&-4&1230&0&&\\
121&-2&0121&-2&1121&{\bf -2}&1221&{\it 2}&1231&{\bf -2}&&\\
122&-3&0122&{\bf -3}&1122&{\it 1}&1222&{\bf 1}&1232&{\it -3}&&\\
123&0&0123&0&1123&0&1223&{\bf 0}&1233&{\it -4}&1234&{\bf 0}\\\hline
\end{array}$$
\caption{Example for Theorem~\ref{alfin}, where the lists corresponding to $\mathcal T_2$, $\mathcal T_3$ and $\mathcal T_4$ are represented according to the respective pairs $(\alpha,h(\alpha))$ indicating column pairs $(\alpha,h(\alpha))$ and $(\alpha_j,h(\alpha_j))$, for $j=1,2,3,4,5$, as shown in the heading line of the table.}
\label{tab5}
\end{table}

\begin{enumerate}
\item the first term of $\Phi_j$ is
\[
\phi_1=\begin{cases}
\mbox{the RGS }1,&\mbox{ if }j=1;\\
\mbox{the smallest RGS with suffix }(j-1)(j-1),&\mbox{ if  }j>1;\\
\end{cases}
\]
\item if $\alpha=a_{k-1}\cdots a_1\in\Phi_j$ and either $a_1=0$ or $a_{k-1}\cdots a_2a'_1\notin\Phi_j$ for every $a'_1<a_1$, then $\alpha|j\in\Phi_j$ for $j\in[0,a_1]$; in that case, if $\alpha_{j'}\in\Phi_j$ with $\alpha_{j'}=a_{k-1}\cdots a_2(a_1+j')$, for $1\le j'\in\mathbb{Z}$, then $\alpha_{j'}|j\in\Phi_j$;
\item for each maximal subsequence $S=(\iota,2,\ldots,2)$ of $i(\mathcal S)$ ($\iota>2$), if there are $z$ penultimate terms $i=2$ of $S$ ($z>0$) heading maximal vertical prefixes of a fixed length $y$ in $h(\Phi_j)$ ($y>0$) and ending at $h(\alpha_j)=h(a_{k-1}\cdots a_3(y+j)y)$ ($j\in[0,z[$), then $\alpha_{j'}=a_{k-1}\cdots a_3(y+z)j'\in\Phi_j$, for $j'\in]y,y+z]$, yielding a vertical suffix $\{h(\alpha_{j'});j'\in]y,y+z]\}$.
\end{enumerate}

\begin{example}
The three rectangular enclosures of Figure~\ref{fig4} contain in left-to-right columnwise form (only showing those columns with yellow squares in Figure~\ref{fig3}) the subsequence $\Phi_1$ of $\mathcal S$ in Subsection~\ref{touse}, (of RGS's $\alpha$ in yellow squares). Such enclosures contain in red the RGS's for the prefixes in item 3 above, and in blue the RGS's for the suffixes.
\end{example}

The columns in the formations of Subsection~\ref{form}, as in Figure~\ref{fig3}, end up with null values $h(\alpha)=0$, which correspond to the terminal nodes $\alpha$ of maximal paths that after their initial nodes $\beta$ with $i(\beta)>1$, have the remaining nodes $\beta'$ with $i=i(\beta')=1$. Clearly, the associated nodes $\alpha$ have degree 1 in the pertaining trees $\mathcal T_k$.

\begin{theorem}\label{0} Let $\alpha$ be a node of $\mathcal T_k$. Then,
\begin{enumerate}
\item if $\alpha$ is a terminal node of a maximal path of $\mathcal T_k$ whose initial node $\beta$ has $i(\beta)>1$ and whose remaining nodes $\gamma$ have $i(\gamma)=1$, then $g(\alpha)=0$;
\item if $\alpha=a_{k-1}\cdots a_1$ with $a_{k-1}=1$ and $a_j=0$, for $j=1,\ldots, k-2$, then $g(\alpha)=0$.
\end{enumerate}
\end{theorem}

\begin{proof} Item 1 in the statement arises because of the presence of the substring $1'1''$ in $F(\alpha)$. Item
 2 arises because of the presence of all substrings $j'j''$ in $F(\alpha)$, for $j=1,\ldots,k-1$.
\end{proof}

\begin{theorem}\label{1} Let $\alpha_1$ be a node of $\mathcal T_k$. Then, $\alpha'_1=1|\alpha_1$ is a node of $\mathcal T_{k+1}$ and
\begin{enumerate}
\item if $h(\alpha_1)\in\Phi_1$, then $h(\alpha'_1)\in\Phi_1$ and $h(\alpha'_1)=k-h(\alpha_1)$;
\item if $h(\alpha_1)\notin\Phi_1$, then $h(\alpha'_1)\notin\Phi_1$ and $h(\alpha'_1)=h(\alpha_1)$.
\end{enumerate}
\end{theorem}

\begin{proof}
Item 1 in the statement occurs exactly when the substring $k'k''$ in $F(\alpha)$ changes position from one side of $1'$ to the opposite side in the procedure of Theorem~\ref{thm1} starting at the parent $\beta$ of $\alpha$ and ending at $\alpha$. Item 2 occurs exactly when that is not the case.
\end{proof}

\begin{example}
Since $\alpha_1=1$ is a node of $\mathcal T_2$ as in Theorem~\ref{0} item 1, then $\alpha'_1=1|\alpha_1=11$ is a node of $\mathcal T_3$ with $h(\alpha'_1)=h(11)=h(1)-k=0-2=-2\in\Phi_1$, by Theorem~\ref{1} item 1. This is indicated by $h(1)=0$ in the upper leftmost yellow square in Figure~\ref{fig3} and its accompanying $h(11)=-2$  as the upper leftmost red integer in the figure. Note that this pattern is continued by associating each yellow square in Figure~\ref{fig3} to a corresponding red integer for all $k<7$. We can annotate this via the successive pairs $(\alpha_1,h(\alpha_1))$ taken by reading the data in Figure~\ref{fig3} from left to right and from top downward:
$$^{(1(0),11(-2)), (10(0),110(-3)),
11(-2),111(1)),(100(0),1100(-4)),(110(-3),1110(1)),(111(1),1111(-3)),(122(-3),1122(1)).}$$ The last pair here arises from $h(122)=-3$, which follows from Corollary~\ref{2}, below.
\end{example}

\begin{theorem}\label{j}
Let $1<j\le k\in\mathbb{Z}$. Let $\alpha_j=1\cdots(j-1)(j-1)a_{k-j-1}\cdots a_1$ be a node of $\mathcal T_k$. Then,
$\alpha'_j=1\cdots (j-1)ja_{k-j-1}\cdots a_1$ is a node of $\mathcal T_k$ and
\begin{enumerate}
\item if $h(\alpha_j)\in\Phi_j$, then $h(\alpha'_j)=k-h(\alpha_j)$;
\item if $h(\alpha_j)\notin\Phi_j$, then $h(\alpha'_j)=h(\alpha_j)$.
\end{enumerate}
\end{theorem}

\begin{proof} Similar to the proof of Theorem~\ref{1}.
\end{proof}

\begin{corollary}\label{2}
Let $1\le k\in\mathbb{Z}$. Let $\alpha_2=11a_{k-3}\cdots a_1$ be a node in $\mathcal T_k$. Then, $\alpha'_2=12a_{k-3}\cdots a_1$ is a node of $\mathcal T_k$ and
\begin{enumerate}\item
if $h(\alpha_2)\in\Phi_2$, then $h(\alpha'_2)=k-h(\alpha_2)$;  \item if $h(\alpha_2)\notin\Phi_2$, then $h(\alpha'_2)=h(\alpha_2)$.\end{enumerate}
\end{corollary}

\begin{example}
Applying Corollary~\ref{2} to $\alpha_2=11,110,111,112$, with respective $h(\alpha_2)=-2,-3,1,0\in\Phi_2$ yields $\alpha'_2=12,120,121,122$ with respective $h(\alpha'_2)=0,0,-2,-3$. In Figure~\ref{fig5}, the RGS's $\alpha_2$ are shown in light-blue squares while the  corresponding RGS's $\alpha'_2$ are shown in yellow squares. Figure~\ref{fig5} extends this coloring for $k\le 7$.
\end{example}

\begin{corollary}\label{3}
Let $1\le k\in\mathbb{Z}$. Let $\alpha_3=122a_{k-4}\cdots a_1$ be a node of $\mathcal T_k$. Then, $\alpha'_3=123a_{k-4}\cdots a_1$ is a node of $\mathcal T_k$ and
\begin{enumerate}
\item if $h(\alpha_3)\in\Phi_3$, then $h(\alpha'_3)=k-h(\alpha_3)$;\item if $h(\alpha_3)\notin\Phi_3$, then $h(\alpha'_3)=h(\alpha_3)$.\end{enumerate}
\end{corollary}

\begin{example}
Applying Corollary~\ref{3} to $\alpha_3=122,1220,1221,1222,1223$ with respective $h(\alpha_3)=-3,-4,2,1,0\in\Phi_3$ yields $\alpha'_3=123,1230,1231,1232,1233$ with respective $h(\alpha'_3)=0,0,-2,-3,-4$. In Figure~\ref{fig5}, the RGS's $\alpha_3$ are shown in thick black while the corresponding RGS's $\alpha'_3$ are shown in thick red. Moreover, Figure~\ref{fig5} extends this font treatment for $k\le 7$. For $k=7$, numbers in Italics in Figure~\ref{fig5} corresponds to members of $\Phi_4$.
\end{example}

Both the integer-valued functions $i=i(\alpha)$ of Theorem~\ref{thm1} and $h=h(\alpha)$ of display (\ref{(1)}) have the same domain, ${\mathcal S}\!\setminus\!\beta(0)$.
A {\it partition of a string} $A$ is a sequence of substrings $A_1,A_2,\ldots,A_n$ whose concatenation $A_1|A_2|\cdots|A_n$ is equal to $A$.

\begin{theorem}\label{alfin} The following items hold. \begin{enumerate}\item[\bf(A)]
The node set of $\mathcal T_{k+1}$ is given by the string $A_k^k=A_1^1|A_2^2|\cdots |A_{k-1}^{k-1}|$ $A_{k-1}^k$, with partition $\{A_1^1,A_2^2,\ldots,A_{k-1}^{k-1},A_{k-1}^k\}$, each $A_i^j$ as a column set  as in Table~\ref{tab2} and Figures~\ref{fig3}--\ref{fig5}, refined by splitting the last column $A_{k-2}^k$ of $A_{k-1}^k$ into the set $B_{k-2}^k$ of its first $k-1$ entries and the set $C_{k-2}^k$ of its last entry, $a_{k-1}a_{k-2}\cdots a_1=12\cdots (k-1)$. The sizes $|A_1^1|$, $|A_2^2|$, $\ldots$, $|A_{k-1}^{k-1}|$, $|B_{k-2}^k|$, $|C_{k-2}^k|$ form the line $\Delta'_{k-1}$ of $\Delta'$.
\item[\bf(B)]
The sequence $h({\mathcal S}\!\setminus\!\beta(0))$ is generated by stepwise consideration of the trees $\mathcal T_{k+1}$, ($1\le k\in\mathbb{Z}$).
In the $k$-th step, the determinations in Theorems~\ref{1} and \ref{j} are to be performed in the natural order of the $(k+1)$-germs $\alpha_j$.  More specifically,
the $k$-step completes those determinations, namely $(\alpha_j,h(\alpha_j))\rightarrow(\alpha'_j,h(\alpha'_j))$,
for the lines of $\Delta'$ corresponding to the sets $A_j^j$  ($j=1,\ldots,k-1$),
and ends up with the determinations $(\alpha_k,h(\alpha_k))\rightarrow(\alpha'_k,h(\alpha'_k))$ in the line corresponding to $B_{k-2}^k$ and $(\alpha_{k+1},h(\alpha_{k+1}))\rightarrow(\alpha'_{k+1},h(\alpha'_{k+1}))$ in the final line, corresponding to $C_{k-2}^k$.\end{enumerate}
\end{theorem}

\begin{proof}
Item (A) represents the set of nodes of $\mathcal T_{k+1}$ via $A_k^k$ and $\Delta'_{k-1}$. This is used in item (B) to express the stepwise nature of the generation of the sequence $h({\mathcal S}\!\setminus\!\beta(0))$. The methodology in the statement is obtained by integrating steps applying Theorems~\ref{1} and \ref{j} in the way prescribed, that yields the correspondence with the lines of $\Delta'$.\end{proof}

\begin{example}
Theorem~\ref{alfin} is exemplified via Table~\ref{tab5}, where the lists corresponding to $\mathcal T_2$, $\mathcal T_3$ and $\mathcal T_4$ are represented according to the respective pairs $(\alpha,h(\alpha))$ indicating column pairs $(\alpha,h(\alpha))$ and $(\alpha_j,h(\alpha_j))$, for $j=1,2,3,4,5$, as shown in the heading line of the table.

The first pair, $(\alpha,h(\alpha))$ shows RGS's $\alpha$ in each case and their corresponding $h(\alpha)$. The following pair, $(\alpha_1,h(\alpha_1))$, shows the $k$-germs $\alpha_1$ corresponding to the RGS's $\alpha$ of the first column and $h(\alpha_1)=h(\alpha)$ but in bold trace if corresponding to a yellow square as in Figure~\ref{fig3}; in that case, the subsequent determinations $(\alpha_1,h(\alpha_1))\rightarrow(\alpha'_1,h(\alpha'_1))$ have the corresponding $h(\alpha'_1)$ in Italics.
This is the case of $h(01)=0$ in bold trace and $h(11)=-2$ in Italics, that we may indicate ``$h(01)=0\rightarrow_1(11)=-2$''.
If a determination $(\alpha_2,h(\alpha_2))\rightarrow(\alpha'_2,h(\alpha'_2))$ happens, then the numbers in Italics are assigned on their right to numbers in bold trace, again.
The cases with bold trace and Italics in Table~\ref{tab5} can then be summarized as follows:
$$\begin{array}{ll}
^{h(01)=0\rightarrow_1h(11)=-2\rightarrow_2h(12)=0,}
_{h(010)=0\rightarrow_1h(110)=-3\rightarrow_2h(120)=0,}&
^{h(011)=-2\rightarrow_1h(111)=1\rightarrow_2h(121)=-2,}
_{h(1000)=0\rightarrow_1h(1100)=-4\rightarrow_2h(1200)=0,}\\
^{h(0110)=-3\rightarrow_1h(1110)=1\rightarrow_2h(1210)=-3,}
_{h(0111)=1\rightarrow_1h(1111)=-3\rightarrow_2h(1211)=1,}&
^{h(1121)=-2\rightarrow_2h(1221)=2\rightarrow_3h(1231)=-2,}
_{h(0122)=-3\rightarrow_1h(1122)=1\rightarrow_2h1222)=1\rightarrow_3h(1232)=-3,}\\
^{h(1223)=0\rightarrow_3h(1233)=-4\rightarrow_4h(1234)=0.}&\end{array}$$
\end{example}

\begin{corollary}\label{coro}
The sequence of pairs $(a(S\!\setminus\!\beta(0)),h(S\!\setminus\!\beta(0)))$ allows to retrieve any vertex $v$ in $O_k$ (resp., $M_k$) by locating its oriented $n$- (resp., $2n$-) cycle in the cycle-factor of \cite{D3,u2f} or in the $\mathbb{Z}_n$- (resp., $\mathbb{D}_n$-) classes as in Section~\ref{s1}, and then locating $v$ departing from the anchored Dyck word in such cycle or class; the sequence also allows to enlist all such vertices $v$ by ordering their cycles (resp., classes), including all vertices in each such cycle (resp.,  class), starting with the corresponding anchored Dyck word.
\end{corollary}

\begin{proof} The function $a(S\!\setminus\!b(0))$, arising from Theorem~\ref{thm1}, yields the required update locations, while the function $h({\mathcal S}\!\setminus\!\beta(0))$ yields the specific updates, as determined in Theorem~\ref{alfin}. This produces the corresponding signatures. Then, Theorem~\ref{id} allows to recover the original Dyck words from those signatures, and thus the vertices of $O_k$ (resp., $M_k$) by local translation in their containing cycles in the mentioned cycle-factors, or cyclic (resp., dihedral) classes.
\end{proof}

\subsection{Asymptotic behavior}\label{Remarque}

It is known that asymptotically the Catalan numbers $C_k$ grow as $\frac{4^k}{k^{\frac{3}{2}}\sqrt{\pi}}$, which is the limit of the single-update process that takes to the determination of all  Dyck words of length $n=2k+1$, as $k$ tends to infinity. By Corollary~\ref{coro}, an orderly determination of all the vertices of $O_k$, resp., $M_k$, is then asymptotically $\frac{4^k}{k^{\frac{3}{2}}\sqrt{\pi}}(2k+1)$, resp., $\frac{4^k}{k^{\frac{3}{2}}\sqrt{\pi}}(4k+2)$.


\begin{thebibliography}{99}

\bibitem{Arndt}
J. Arndt,
Matters Computational: Ideas, Algorithms, Source Code,
Springer, 2011.

\bibitem{B} N. Biggs,
{\it Some odd graph theory},
Annals of the New York Academy of Sciences,
{\bf 319} (1979),
71--81.

\bibitem{D2}
I. J. Dejter,
{\it A numeral system for the middle-levels graphs},
Electronic Journal of Graph Theory and Applications
{\bf 9} (2019),
137-156.

\bibitem{D1} I. J. Dejter,
{\it Reinterpreting the middle-levels theorem via natural enumeration of ordered trees},
Open Journal of Discrete Applied Mathematics,
{\bf 3} (2020),
8--22.

\bibitem{D3}
I. J. Dejter,
{\it Arc coloring of odd graphs for hamiltonicity},
Open J. Discrete Appl. Math., {\bf(6)} (2023), no 2, 14--31.

\bibitem{gmn}
P. Gregor, T. M\"utze and J. Nummenpalo,
{\it A short proof of the middle levels theorem},
Discrete Analysis,
{\bf 8} (2018),
12pp.

\bibitem{M}
T. M\"utze,
{\it Proof of the middle levels conjecture},
Proc. of the London Mathematical Society,
{\bf112} (2016),
677--713.

\bibitem{Hcs}
T. M\"utze, J. Nummenpalo and B. Walczak,
{\it Sparse Kneser graphs are hamiltonian},
Journal of the London Mathematical Society,
{\bf 103} (2021),
1253--1275.

\bibitem{u2f}
T. M\"utze, C. Standke, and V. Wiechert,
{\it A minimum-change version of the Chung--Feller theorem for Dyck paths},
European J. Combin.,
 {\bf 69} (2018),
 260--275.

\bibitem{oeis}
N.\ J.\ A.\ Sloane,
The On-Line Encyclopedia of Integer Sequences,
\url{https://oeis.org/}.

\bibitem{Stanley}
R. Stanley,
Enumerative Combinatorics, Volume 2,
Cambridge University Press, 1999.

\end{thebibliography}
\end{document}